\documentclass[11pt,oneside]{amsart}
  \usepackage{amsmath,amssymb,amsthm}
  \usepackage{geometry}
  \usepackage{accents}
  \usepackage{graphicx}
  \usepackage{enumerate}
  \usepackage{color}
  \usepackage{xspace,verbatim}
 \usepackage{hyperref}
\input xy
\xyoption{all}

\def\R{\mathbb{R}}
\def\Q{\mathbb{Q}}
 \def\Z{\mathbb Z}

\def\Hom{\mbox{Hom}}

\def\sF{\mathcal{F}}
\def\sC{\mathcal C}

\def\sI{\mathcal I}

\def\sL{\mathcal L}
\def\sP{\mathcal P}

\def\sG{\mathcal G}

\def\MCG{\mbox{MCG}}

\def\sM{\mathcal M}

\def\tt{\tau}

\def\h{\mathfrak h}

\def\j{\mathfrak j}

\def\r{\mathfrak c}
 
\DeclareMathAlphabet{\mathpzc}{OT1}{pzc}{m}{it}

\def\mod{\mbox{mod}}

\def\wzeta{\widetilde \zeta}

\newtheorem{theorem}{Theorem}[section]
\newtheorem{corollary}[theorem]{Corollary}
\newtheorem{lemma}[theorem]{Lemma}
\newtheorem{proposition}[theorem]{Proposition}
\newtheorem{definition}[theorem]{Definition}
\newtheorem{remark}[theorem]{Remark}
\newtheorem{question}[theorem]{Question}

\newtheorem{example}[theorem]{Example}
  \newtheorem{introthm}{Theorem}

\begin{document}
\title{Quotient families of mapping classes}
\author{Eriko Hironaka} 
\thanks{\small \noindent 
This work was partially supported by a grant from the Simons Foundation \#426722.}
\date{}

\begin{abstract}  In this paper, we define {\it quotient families} of mapping classes generalizing an example of Penner. We show that these 
mapping classes are  contained in a single flow-equivalence class of monodromies of a fibered 3-manifold $M$.    
The special structure of quotient families helps 
to compute global invariants of $M$ such as the  Alexander polynomial, and (in the case when $M$ is hyperbolic) the Teichm\"uller 
poliynomial of the associated fibered face.  These in turn give information about the homological and geometric dilatations of the
mapping classes in the quotient family.
\end{abstract}

\subjclass{14J50, 37F15, 57M27}
\maketitle

\section{Introduction}

In \cite{Penner91} Penner constructed a sequence of pseudo-Anosov mapping classes, sometimes called {\it Penner wheels},
with asymptotically small dilatations. 
In this paper we define a generalization of Penner wheels called {\it quotient families}, and put them in  the framework of the Thurston-Fried-McMullen
fibered face theory \cite{Thurston:norm} \cite{Fried82} \cite{McMullen:Poly}.   Specifically, we show that each  quotient family corresponds naturally to a
linear segment of a  fibered face of a  3-manifold.   Putting quotient families in the fibered face context 
helps to determine their Nielsen-Thurston classification, and in the pseudo-Anosov case makes it possible to compute dilatations via the Teichm\"uller polynomial.

\subsection{Pseudo-Anosov mapping classes, dilatations, and fibered faces}
Let $S$ be a connected oriented surface of finite type with negative Euler characteristic $\chi(S)$. A {\it mapping class} $\phi : S \rightarrow S$ is an orientation preserving homeomorphism modulo isotopy (both the mapping class and the isotopies are assumed to fix the boundary if $S$ has boundary). 
The Nielsen-Thurston classification states that mapping classes are either periodic,  reducible, or pseudo-Anosov, where $(S,\phi)$ is
{\it  pseudo-Anosov} if $\phi$ preserves a pair of transverse measured singular {\it stable} and {\it unstable} foliations $(\mathcal F^\pm, \nu^\pm)$ 
and $\phi^*(\nu^\pm) =  \lambda^{\pm 1} \mu^\pm$ for some $\lambda > 0$ \cite{Thurston88}.    The constant 
$$
\lambda(\phi) = \lambda
$$
is uniquely determined by $(S,\phi)$
and is called the  {\it dilatation} of $\phi$.  The singularities of $\mathcal F^\pm$ are called the
the {\it singularities of $\phi$}.  See, for example,  \cite{FM:MCG} for more details.

In \cite{Penner91}, Penner  constructed a sequence of
pseudo-Ansoov mapping classes  
 $(R_g,\psi_g)$, for $g \ge 3$, where $R_g$ is a closed surface of genus $g \ge 2$ and $\lambda(\psi_g)^g \leq 11$.  Using this he showed that  the minimum expansion factor $l_g$ for a pseudo-Anosov mapping
classes of on a genus $g$ surface behaves asymptotically like $\log(l_g) \asymp \frac{1}{g}$ as a function of $g$, that is, there is a constant $C \ge 1$ so that
$$
\frac{1}{C g}\ \leq\ \log (l_g) \ \leq\  \frac{C}{g}.
$$

Since then the Thurston-Fried-McMullen fibered face theory (recalled below) has been used to show that pseudo-Anosov mapping classes with bounded
{\it normalized dilatation}  $L(S,\phi) = \lambda(\phi)^{|\chi(S)|}$ are naturally grouped together into dynamical families: 
these arise as subcollections of
monodromies of  hyperbolic fibered 3-manifolds that have first Betti number greater than or equal to two \cite{McMullen:Poly},
and are parameterized by rational points on compact subsets of fibered faces.
Furthermore, the Farb-Leininger-Margalit Universal Finiteness Theorem implies that
any family of pseudo-Anosov mapping classes with bounded normalized dilatations is contained in the
set of monodromies of a finite set of fibered 3-manifolds  up to fiber-wise Dehn fillings \cite{FLM09}.

\subsection{Penner wheels} 
Penner wheels are defined as follows.  
Consider the genus $g$ surface $R_g$ as a surface with rotational symmetry of order $g$ fixing two points, as
drawn in Figure~\ref{Pennerexample-fig}, and 
let  $\zeta_g$ be the counterclockwise rotation by the angle
$\frac{2\pi}{g}$.
For a simple closed curve $\gamma$ on a surface, let $\delta_{\gamma}$ be the right Dehn twist centered at  $\gamma$.  
Let $\eta_g =  \delta_{c_g}  \delta_{b_g}^{-1} \delta_{a_g}$ be the product of Dehn twists centered along the labeled curves
$a_g,b_g,c_g$ drawn in Figure~\ref{Pennerexample-fig}.   Then Penner's sequence consists of the pairs $(R_g,\psi_g)$, where $\psi_g = \zeta_g  \eta_g$.
\begin{figure}[htbp]
\begin{center}
\includegraphics[width=1.5in]{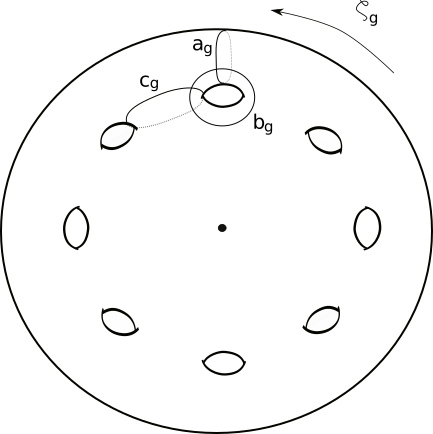}
\caption{Penner wheel on a surface of genus $g$ with rotational symmetry fixing two points, one of which is marked in the figure as a central dot. The other is the image of this point under front to back symmetry.}
\label{Pennerexample-fig}
\end{center}
\end{figure}

\subsection{Quotient families}\label{quot-sec}
To define quotient families, consider a triple $(\widetilde S,\widetilde \zeta,\widetilde \eta)$, where $\widetilde S$ is an oriented surface of infinite
type and 
$$
\widetilde \zeta, \widetilde \eta: \widetilde S \rightarrow \widetilde S
$$
are homeomorphisms satisfying the following:
\begin{enumerate}
\item $\widetilde \zeta$ generates a  properly discontinuous, orientation-preserving, fixed-point free,  infinite cyclic action on
$\widetilde S$;
\item $\widetilde S/\widetilde \zeta$ is a surface of finite type; 
\item the action of $\widetilde\zeta$ has a fundamental domain $\Sigma_0$, a compact, connected, oriented surface of finite type
with boundary $\partial \Sigma_0$, satisfying
\begin{enumerate}[(i)]
\item $\widetilde \zeta(\Sigma_0) \cap \Sigma_0 \subset \partial \Sigma_0$, and
\item  $\widetilde \zeta^2(\Sigma_0) \cap \Sigma_0 = \emptyset$;
\end{enumerate}
and
\item the support of $\widetilde\eta$ is strictly
contained in
$$
\Sigma_0 \cup \widetilde \zeta \Sigma_0 \cup \cdots \cup \widetilde\zeta^{m_0} \Sigma_0
$$
for some finite $m_0$.
\end{enumerate}
We say that the triple $(\widetilde S, \widetilde\zeta,\widetilde \eta)$ forms a {\it template} of {\it width} $m_0$.

\begin{figure}[htbp] 
   \centering
   \includegraphics[width=2.5in]{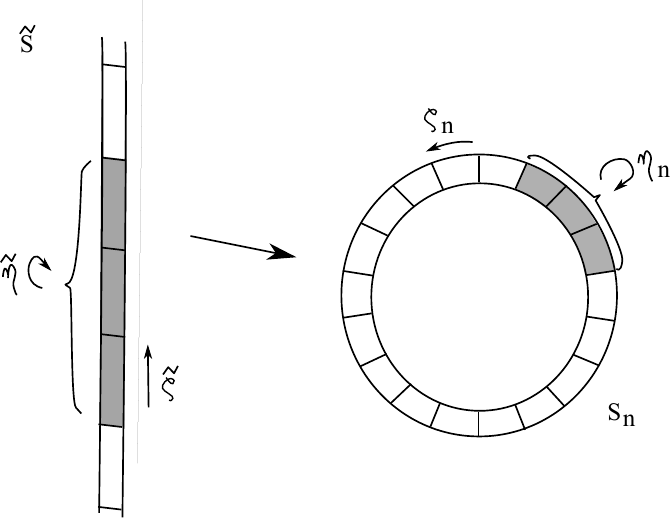} 
   \caption{This is an illustration of a template $(\widetilde S,\widetilde \zeta, \widetilde \eta)$ of width $m_0=2$, and the quotient mapping
   class $(S_n,\zeta_n^k \circ \eta_n)$. The supports of 
   $\widetilde \eta$ on $\widetilde S$, and $\eta_n$ on $S_n$ are shaded.}
   \label{bamboo-fig}
\end{figure}

Let $I_{m_0}(\Q)$ be the rational points on the open interval $I_{m_0} = (0,\frac{1}{m_0})$. 
From a template $(\widetilde S,\widetilde \zeta,\widetilde \eta)$ with width $m_0$ we define an associated {\it quotient family} 
$Q(\widetilde S,\widetilde \zeta,\widetilde \eta)$
parameterized
by  $I_{m_0}(\Q)$.
 For $\r \in I_{m_0}(\Q)$, where $\r = \frac{k}{n}$ is in reduced form,  define a mapping class 
 $\widetilde \phi_\r :\widetilde S \rightarrow \widetilde S$
 as follows.
Let $\widetilde \eta_n$ be the composition
$$
\widetilde \eta_n = \underset{r \in \Z}\circ {\widetilde \zeta}^{rn}\widetilde \eta {\widetilde \zeta}^{-rn}.
$$
(See illustration in Figure~\ref{bamboo-fig}).
This is well-defined on $\widetilde S$ since $n > m_0$ implies that the supports of
${\widetilde \zeta}^{rn}\widetilde \eta {\widetilde \zeta}^{-rn}$ are disjoint for distinct $r$.

Let
 $$
 \widetilde \phi_{\r} = \widetilde \zeta^{\overline k_n} \circ \widetilde \eta_n.
 $$
 where $\overline k_nk = 1 (\mod \ n)$.  
Since $\widetilde \eta_n$ is invariant under conjugation by $\widetilde \zeta^n$, it defines a well-defined homeomorphism
$\eta_n$ on the quotient space $S_{\r}=\widetilde S/\zeta^n$.  Similarly, $\widetilde \zeta$ defines a homeomorphism
$\zeta_n$ on the quotient space $S_{\r}$.   Let
 $$
 \phi_{\r} = (\zeta_n)^{\overline k_n} \circ \eta_n.
 $$
Then $\widetilde \phi_{\r}$ is a lift of $\phi_{\r}$ by the covering map $\widetilde S \rightarrow S_{\r}$.

The {\it quotient family} associated to the template $(\widetilde S,\widetilde \zeta,\widetilde \eta)$ is
defined by
$$
Q(\widetilde S,\widetilde \zeta,\widetilde \eta) = \{(S_{\r},\phi_{\r})\ |  \ \r \in I_{m_0}(\Q)\}.
$$

\begin{example}[Penner Wheel as a sequence in a quotient family]{\em
Let $R_g^0$ be the surface of genus $g$ and two boundary components obtained by
 removing two disks on $R_g$
surrounding the fixed points of $\zeta_n$ and let $\phi_g^0$ be the
restriction of $\phi_g$ on $R_g^0$.  Let $\widetilde S$ be as in Figure~\ref{original_lifts-fig}, a homeomorphism $\zeta$
acting as vertical translation on a fundamental domain $\Sigma_0$ and its orbits.  Let 
$\widetilde a$, $\widetilde b$ and $\widetilde c$ be distinguished points on $\widetilde S$,
and let $\widetilde \eta  =  \delta_{\widetilde c}  \delta_{\widetilde b}^{-1} \delta_{\widetilde a}$.
\begin{figure}[htbp]
\begin{center}
\includegraphics[width=3in]{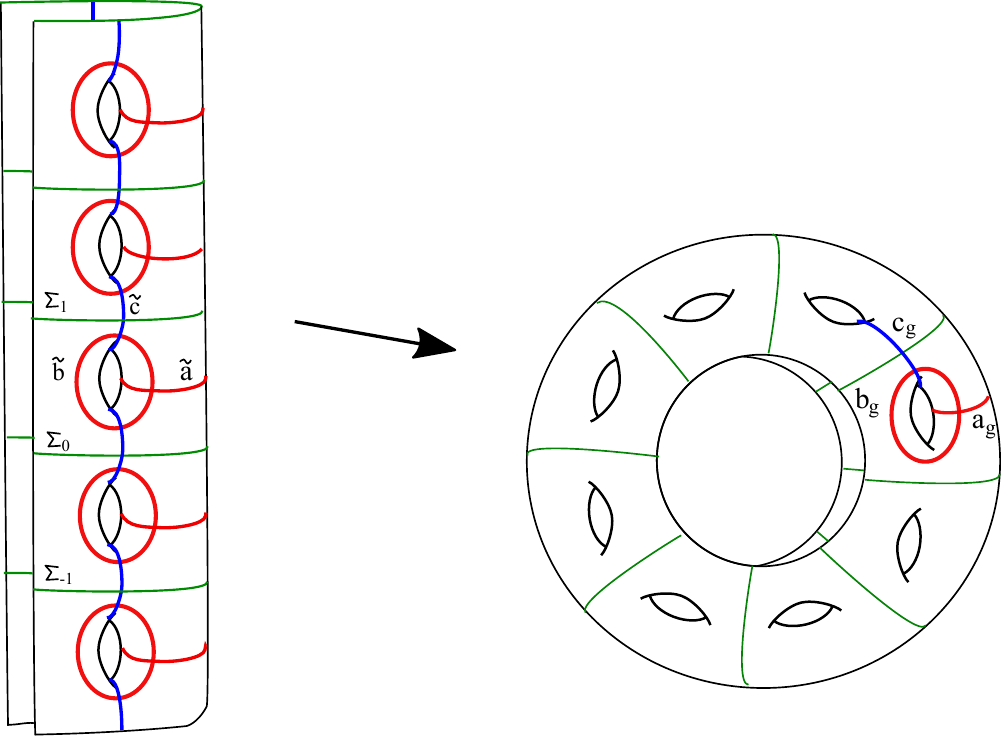}
\caption{Cyclic covering $\widetilde S$ and quotient surface $S_g$ for Penner's example.}
\label{original_lifts-fig}
\end{center}
\end{figure}
Then $(R_g^0, \phi_g^0)$, $g \ge 2$ is a sequence in the quotient family associated to $(\widetilde S,\widetilde \zeta, \widetilde \eta)$  associated
to the sequence $\frac{1}{g} \in \sI_1(\Q)$.
}
\end{example}

\subsection{Fibered faces and parameterizations of flow-equivalence classes}\label{fibered-sec}
Let $\MCG(S)$ be the group of mapping classes defined on $S$.
Fibered face theory gives a way to partition the set of all mapping classes 
$$
\MCG = \{(S,\phi) \ : \ 
\phi \in \MCG(S), \mbox{$S$ a connected orientable surface of finite type}\}
$$
into families with related dynamics.   Each mapping class $(S,\phi)$  defines 
\begin{enumerate}
\item a 3-manifold $M$ with mapping torus structure $M =  [0,1] \times S/{(x,1) \sim (\phi(x),0)}$;
\item a distinguished fibration $\rho: M \rightarrow S^1$ defined by projection onto the
second coordinate with {\it monodromy}  $(S,\phi)$; and 
\item  a one-dimensional oriented {\it suspension flow}, or {\it foliation with oriented leaves},
$\sL$ on $M$ whose leaves are the images of the leaves $\R \times \{x\}$ under the cyclic covering map
$\R \times S \rightarrow M$ corresponding to the kernel of the map $ \rho_* : \pi_1(M) \rightarrow \Z$.
\end{enumerate}
Two mapping classes are said to be {\it flow-equivalent}  if they define the same pair $(M,\sL)$.
The induced homomorphism $\rho_* : H_1(M;\Z) \rightarrow \Z$ defines an element $\alpha\in H^1(M;\R)$, called a {\it fibered element}.
 
In \cite{Thurston:norm} Thurston defines a semi-norm $|| \ ||$ on $H^1(M;\R)$ with a convex polygonal unit norm ball,
with compact closure if  $M$ is hyperbolic.
Each cone $V_F$ over an open top-dimensional face $F$ is either {\it fibered} (all primitive integral elements are fibered) or contains  no fibered elements.   Since there is one primitive integral element on the ray through each rational points on the fibered face $F$, it
 follows  that 
the union of rational points on fibered faces of oriented 3-manifolds has a natural surjection onto the set of all mapping classes on oriented surfaces of
finite type.
Explicitly, for  $\alpha$ a fibered element in a fibered cone $V_F$, let $\overline\alpha$ be its projection onto $F$ along the rational ray,
and let $(S_\alpha,\phi_\alpha)$ be its monodromy.
For each fibered 3-manifold $M$ and fibered face $F \subset H^1(M;\R)$
 define
\begin{eqnarray*}
\mathfrak f_F: F(\Q) &\rightarrow& \MCG\\
\overline \alpha &\mapsto& (S_\alpha,\phi_\alpha)
\end{eqnarray*}
taking each $\overline \alpha$ to the monodromy $(S_\alpha,\phi_\alpha)$, where $\alpha \in V_F$ is 
the primitive integral element that is a positive multiple of $\overline \alpha$.

Let $\sC$ be the set of all fibered faces, and let $\mathfrak C$ be the set of all flow-equivalence classes.
Then we have a surjection
$$
\mathfrak f: \sC \rightarrow \mathfrak C
$$
taking fibered faces to corresponding flow-equivalence classes.

Our first result shows that quotient families have natural parameterizations by linear segments on a fibered face.

\begin{introthm}\label{quotient-thm}  Each quotient family $Q$ is contained in some flow equivalence class $\mathfrak F$.  Let
$F$ be a fibered face with $\r(F) = \mathfrak F$.  Then there is an embedding 
$$
\iota : I_{m_0} \hookrightarrow F,
$$
such that 
\begin{enumerate}
\item the image $\iota(I_{m_0})$ is a linear section of $F$ in $H^1(M;\R)$,
\item $\iota$ restricts to a map $I_{m_0}(\Q) \rightarrow F(\Q)$,  and
\item for all $\r \in I_{m_0}(\Q)$, 
$$
(S_{\r},\phi_{\r}) = \mathfrak f_F(\iota(c)).
$$
\end{enumerate}
\end{introthm}

Now consider the set of all pseudo-Anosov mapping classes on all oriented surfaces of finite type (possibly with boundary and/or punctures)
$\sP \subset \MCG$.   By a result of Thurston \cite{Thurston88}
the (interior of the) mapping torus $M$ of $(S,\phi)$ is hyperbolic if and only if $(S,\phi)$ is pseudo-Anosov.  

Thus, Theorem~\ref{quotient-thm} has the following immediate corollary:

\begin{corollary}\label{quotient-cor} A quotient family  $Q$ is contained in $\sP$ if and only if 
 $Q \cap \sP \neq \emptyset$.
\end{corollary}

\subsection{Fibered faces and  bounded normalized dilatations}
Let $M$ be a hyperbolic 3-manifold with fibered face $F \subset H^1(M;\R)$.  Fried \cite{Fried82} showed that the function
$$
\alpha \mapsto \log \lambda(\phi_\alpha)
$$
defined for $\alpha$ a primitive integral element of the cone $V_F = F \cdot \R^+$ extends to a continuous  convex 
function on $V_F$ that is homogeneous of
degree $-1$ and goes to infinity toward the boundary of $V_F$.  The Thurston norm $|| \ ||$ 
on $H^1(M;\R)$ has the property that $|| \alpha|| = |\chi(S_\alpha)|$
for all integral elements $\alpha$ on a fibered cone \cite{Thurston:norm}.
Noting that normalized dilatation is the post-composition of Fried's function with the
exponential function, we have the following.

\begin{theorem} [Fried \cite{Fried82}]\label{Fried-thm}  Given a flow-equivalence class $\mathfrak F \subset \sP$ and fibered face $F$ with $\mathfrak c(F) = \mathfrak F$, the normalized dilatation function $L$ extends uniquely to 
a continuous function
$$
L :  V_F \rightarrow \R
$$
that is constant on rays through the origin and on $F$ it defines a convex function that goes to infinity toward the boundary of $F$.
\end{theorem}

\noindent
By this theorem, it suffices to think of $L$ as a function on $F$ thought of as the quotient of $V_F$ by positive scalar multiplication.

Farb-Leininger-Margalit's  {\it Universal Finiteness Theorem} states
conversely that for any $C > 0$, there is a finite set of fibered 3-manifolds $\Omega_C$ such that
 for any pseudo-Anosov map $(S,\phi)$ with $L(S,\phi) < C$, there is some $M \in \Omega_C$ such that
 $M$ is the mapping torus for $(S^0,\phi^0)$, where  $(S^0,\phi^0)$ is the {\it fully-punctured} mapping class 
 obtained from $(S,\phi)$ by removing the singularities of the $\phi$.  \cite{FLM09}

\subsection{Behavior of normalized dilatations and stability}

Our second result deals with the behavior of the normalized dilatations of a quotient family with pseudo-Anosov elements.
We say $Q = Q(\widetilde S,\widetilde \zeta,\widetilde \eta)$ is a {\it stable} family if there are integers $m_1 >0$ such that for all
for $x \in \Sigma_0$ and $m > m_1$ we have
$$
(\widetilde \zeta \widetilde \eta)^{m+1} (x) = \widetilde \zeta (\widetilde \zeta \widetilde \eta)^m (x),
$$
and $Q$ is {\it bi-stable} if both $Q(\widetilde S,\widetilde \zeta, \widetilde \eta)$ and $Q(\widetilde S, \widetilde \zeta^{-1}, \widetilde \eta^{-1})$
are stable.

Consider the function defined by
\begin{eqnarray*}
\widetilde \phi : \widetilde S &\rightarrow& \widetilde S\\
x& \mapsto& \widetilde \zeta^{r-m_1}(\widetilde \zeta\widetilde \eta)^{m_1}\widetilde \zeta^{-r} (x),
\end{eqnarray*}
where $r$ is any integer such that $\widetilde \zeta^{-r}(x) \in \Sigma_0$.  In Lemma~\ref{welldefinedhomeo-lem} we show
that  if $Q$ is stable, then $\widetilde \phi$ is a continuous open map, and if $Q$ is bi-stable, then $\widetilde \phi$ is a homeomorphism.

 \begin{introthm}\label{stability-thm}  If $Q$ is bi-stable, then we have the the following:
 \begin{enumerate}
 \item the map $\widetilde \phi$ defines a mapping class on $\widetilde S$ that commutes with the action of $\Z$, and hence defines a mapping class
$(S_0,\phi_0)$, where $S_0 = \widetilde S /\widetilde \zeta$ and $\phi_0$ is the mapping
 class on $S$ induced by $\widetilde \phi$;
 \item the  map $\iota : I_{m_0} \rightarrow F$ extends to the half open interval $[0,\frac{1}{m_0})$
 and the monodromy at $0$ equals $(S_0,\phi_0)$; and
 \item \label{convergence-item}  the value of the normalized stretch-factor  $L(S_{\r},\phi_{\r})$ converges to a value $L_0$
 with 
 $$
 1 < L_0 < \infty
 $$
as $\r$ approaches $0$.
 \end{enumerate}
 If $Q$ is not stable, then the map $\iota(t)$ converges to a point on the boundary of $F$ as $t$ approaches $0$.
 \end{introthm}
 
 \noindent
 By Fried's theorem,  we have the following immediate corollary.
 
 \begin{corollary} If $Q$ is bi-stable then, 
 $$
 \lim_{c \rightarrow 0} L(S_{\r},\phi_{\r}) = L(S_0,\phi_0);
 $$
and if $Q$ is not stable, then $\lim_{c \rightarrow 0}\iota(c)$ lies on the boundary of $F$ and
 $$
 \lim_{c \rightarrow 0} L(S_{\r},\phi_{\r}) = \infty.
 $$
 \end{corollary}

\begin{remark} {\em It is not known to the author whether stability implies bi-stability for $Q$.  However, bi-stability is a necessary condition for 
$\widetilde \phi$ to be invertible.}
\end{remark}

\begin{remark} {\em Up to now explicit examples and partial generalizations of Penner wheels have been studied without putting them in the context of fibered faces (see   \cite{Bauer:thesis} \cite{Tsai08} \cite{Valdivia:gn}).  One benefit of seeing quotient families as elements of  a single fibered face
is the possibility of getting explicit defining equations for the geometric and homological stretch-factors via the  Teichm\"uller and Alexander polynomials.
We carry out some calculations in Section~\ref{example-sec}.
}\end{remark}

\subsection{Idea of proofs}
In \cite{McMullen:Poly}, McMullen makes the key observation that 
given a fibered hyperbolic 3-manifold $M$ with monodromy $\phi : S \rightarrow S$
and pseudo-Anosov flow $\sL$, one can study the transverse measures on $\sL$ defined by points on  
the associated fibered face by lifting $\widetilde \sL$ to the maximal abelian covering $\widetilde M^{\mbox {ab}}$ of $M$.

Topologically one
can think of ${\widetilde M}^{\mbox{ab}}$ as a product 
$$
{\widetilde M}^{\mbox{ab}} = \R \times \widetilde S
$$
for a surface $\widetilde S$ of infinite type, where
the topologically trivial foliation $\widetilde \sL$ with leaves 
$$
\{\R \times \{x\} \ : \ x \in\widetilde S\}
$$
is a lift of $\widetilde \sL$.
The subtle geometric information  associated to a point on the fibered cone, such as 
invariant transverse measure on $\sL$ and expansion factor of the monodromy action, can be 
translated to information about the action of the covering automorphism group $H^{\mbox{ab}}$ on the module of transversals of 
$\widetilde \sL$.

In this paper, we make use of this idea, but in reverse.  To build the desired 3-manifold $M$, a foliation $\sL$ and 
a fibered face $F$ from a template $(\widetilde S, \widetilde \zeta, \widetilde \eta)$ we begin with the 3-manifold of infinite type: $\widetilde M' =
\R \times \widetilde S$, a fixed point free, properly discontinuous action of a rank-2 free abelian group $H' = \langle T', Z'\rangle$ on $\widetilde M'$
that respects the product structure
and a foliation $\widetilde \sL'$ preserved by $H'$.
We know that topologically $\widetilde \sM'$  would have to be homeomorphic to an abelian covering of $M$, 
but the group action $H'$ and the foliation $\widetilde \sL'$ will typically need to be adjusted.
To do this we use cutting and pasting on $\widetilde M'$
to create a new (though homeomorphic) $\widetilde M$, and an adjusted foliation $\widetilde \sL$ and covering group $H$.  Then $M = \widetilde M/H$ and
the quotient $\R^H$ of $H_1(M;\R)$ defines a 2-dimensional dual subspace  $W \subset H^1(M;\R)$.  To prove Theorem~\ref{quotient-thm} we define 
an inclusion
$$
\iota: I_{m_0}\hookrightarrow F \cap W
$$
 where $F$ is a fibered face of $M$, so that each mapping class in $Q$  parametrized by a rational
point of $I_{m_0}$ is the monodromy parameterized by its image in $F$ under $\iota$.

By fibered face theory and Theorem~\ref{quotient-thm}
there are two possible behaviors for the normalized dilatations of $(S_{\r},\phi_{\r})$ as $\r \rightarrow 0$ depending on whether $\iota(c)$ converges
to the boundary of $F$ (implying unboundedness) or to an interior point of $f$ as $\r$ approached $0$ (implying convergence).  
In Section~\ref{stable-sec},  we prove Theorem~\ref{stability-thm} by showing that $Q$ is stable if and only if it is possible to find a quotient mapping
class corresponding to the lower limit point of $I_{m_0}$ such that $\iota$ extends.  
In Section~\ref{example-sec}, we illustrate how the results can be applied to give explicit computations of invariants for a given
quotient family, namely two-variable specializations of the Alexander polynomial of $M$ and the Teichm\"uller polynomial for
$F$.

\subsection{Acknowledgements.} Work on this paper began during a semester long visit spent at Tokyo Institute of Technology in 2012.
The author is grateful for the stimulating discussions and support of the department members during this period.  
Since then, this paper has gone through several revisions. We thank
B. Farb, S. Fenley,  C. McMullen, H. Sun,  and D. Valdivia for their helpful comments along the way, 
and the anonymous referee whose careful reading and valuable comments and
corrections have led to a substantially improved exposition.

\section{Construction of a one-dimensional flow-equivalence class}
 
We begin this section by describing our construction in Section~\ref{construction-sec} starting with 
a template $(\widetilde S,\widetilde \zeta,\widetilde \eta)$ and using cutting and pasting on $\R \times \widetilde S$ 
to build an $H =\Z \times \Z$-covering $\widetilde M$ of our desired 3-manifold $M$, and a one-dimensional foliation $\widetilde \sL$ on $\widetilde M$.  
Using this we prove Theorem~\ref{quotient-thm}  in Section~\ref{quotient-thm-sec} and Theorem~\ref{stability-thm} in 
Section~\ref{stable-sec}.

\subsection{Building a manifold with a flow using coverings}\label{construction-sec}
Our construction has three parts.  We first build 
a topological model $\widetilde M'$ for the abelian covering of $M$, and a projection
$\h' : \widetilde M' \rightarrow \R \times \R$ (Section~\ref{setup-sec}).The map $\h'$ allows us to specify a cutting locus whose 
connected components are
preimages under $\h'$ of connected line segments in $\R \times \R$ (Section~\ref{cutting-sec}).  The cutting locus
is re-pasted together using mappings defined by $\widetilde \eta$ (Section~\ref{pasting-sec}).   
This results in a new 3-manifold $\widetilde M$ with a one-dimensional foliation $\widetilde \sL$ and an $H = \Z \times \Z$
 action that preserves the leaves of $\widetilde \sL$  and makes $\widetilde M$ a covering of $M = \widetilde M/H$.  The map 
 $\h'$ defines a function $\h : \widetilde M \rightarrow \R \times \R$ with jump discontinuities at the cutting loci.  This will be usefuls
 for describing and 
 visualiziung the cross-sections of $\widetilde \sL$, 
 corresponding to rational points in $I_{m_0}$, 
 all of which are homeomorphic to $\widetilde S$ but embedded
 differently in $\widetilde M$.

\subsubsection{Set up} \label{setup-sec}
Let $\Sigma_0$ be a fundamental domain on $\widetilde S$ for the map $\widetilde \zeta$, with the following properties: 
\begin{figure}[htbp] 
   \centering
   \includegraphics[width=2in]{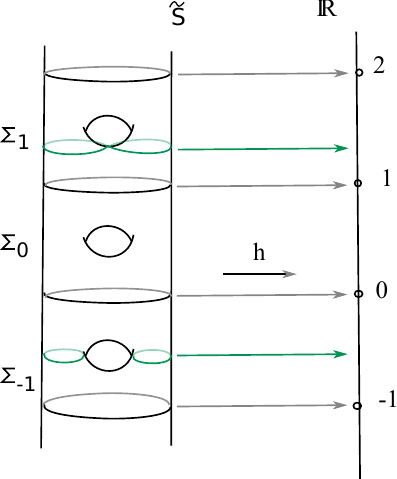} 
   \caption{This is an example of a height function $h$ on $\widetilde S$ that commutes with the action of $\widetilde \zeta$. In this example, 
   the fundamental domain $\Sigma_0$ and all
   its images under powers of $\widetilde\zeta$ have genus one and two boundary components.}
   \label{infinite-fig}
\end{figure}
\begin{enumerate}[(i)]
\item  $\Sigma_0$ is a connected closed submanifold of $\widetilde S$ (with boundary); 
\item  $\widetilde \zeta(\Sigma_0) \cap \Sigma_0$ is a finite disjoint union of closed arcs on  the boundary $\partial \Sigma_0$ of $\Sigma_0$; 
\item $\widetilde \zeta^2(\Sigma_0) \cap \Sigma_0 = \emptyset$; and
\item $\widetilde S = \bigcup_{i \in \Z} \zeta^i(\Sigma_0)$.
\end{enumerate}
Let $\widetilde M' = \R \times \widetilde S$,  let $\widetilde \sL'$ be the trivial 1-dimensional foliation defined on $\widetilde M'$ with 
oriented leaves 
$$
\{\R \times \{y\} \ | \ y \in \widetilde S\}.
$$

The manifold $\widetilde M'$ comes with a natural $H' = \Z \times \Z$ action generated by the
orientation preserving homeomorphisms $T'$ and $Z'$ defined by
\begin{eqnarray*}
T' : \widetilde M' &\rightarrow& \widetilde M'\\
(t,x) &\mapsto& (t-1,x)
\end{eqnarray*}
and
\begin{eqnarray*}
Z' : \widetilde M' &\rightarrow& \widetilde M'\\
(t,x) &\mapsto& (t,\widetilde \zeta(x)).
\end{eqnarray*}
Note that $T'$ and $Z'$ commute, act freely and properly discontinuously on $\widetilde M'$, and preserve
 the leaf structure of $\widetilde \sL'$.
 
We now choose a continuous projection $\h' : \widetilde M' \rightarrow \R \times \R$ so
that powers of $T'$ (and $Z'$) correspond to integer shifts in the first (and second) coordinate.
Let $h$ be a continuous surjective function
$h : \widetilde S \rightarrow \R$  (see Figure~\ref{infinite-fig} for an illustration):
\begin{enumerate}
\item each fiber of $h$ is an immersed union of simple closed curves (rel. punctures) on $\widetilde S$
that split $\widetilde S$ into exactly two pieces; 
\item $\Sigma_0 \cap \widetilde \zeta^{-1}(\Sigma_0) = h^{-1}(0)$; and 
\item $h(\widetilde \zeta^k(x)) = h(x) + k$
for all $x \in \widetilde S$ and $k \in \Z$.
\end{enumerate}
Let $\h' : \widetilde M' \rightarrow \R \times \R$ be defined by $\h' =  \mbox{id} \times h$.

\subsubsection{Cutting}\label{cutting-sec}
We construct a cutting locus $\mathcal G'$ on $\widetilde M'$ by first defining 
a locus on $\R \times \R$, and taking
the preimage by the map $\h'$. 
Let
 $\Gamma_{0,0} \subset \R \times \R$ be the straight line segment connecting $(0,0)$ to $(\frac{1}{2}, m_0)$,     
and, for $(a,b) \in \R\times \R$, let $\Gamma_{a,b} = (a,b) + \Gamma_{0,0}$ be the parallel translate of $\Gamma_{0,0}$ by $(a,b)$ 
(see Figure~\ref{cutsurface-fig}).
Let $X'_{a,b} = {\h'}^{-1}(\Gamma_{a,b})$.

\begin{figure}[htbp] 
   \centering
   \includegraphics[width=2in]{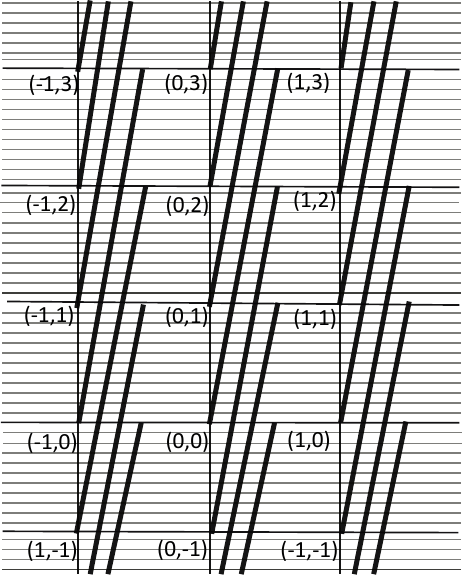} 
   \caption{The image of the cutting loci $\Gamma_{a,b}$ in $\R \times \R$. Horizontal lines indicate the flow $\h'(\widetilde\sL')$.}
   \label{cutsurface-fig}
\end{figure} 

Then we have the following:
\begin{enumerate}
\item the projection $\sigma' : \widetilde M' \rightarrow \widetilde S$ to the second coordinate defines identifications
$$
\sigma'_{(a,b)} : X'_{a,b} \overset\sim\longrightarrow  \widetilde \zeta^{b} \Sigma_0 \cap \cdots \cap \widetilde \zeta^{m_0+ b} \Sigma_0
$$
for each $(a,b) \in \Z \times \Z$;
\item the maps $T'$ and $Z'$ satisfy
$$
\h' (T' (x,y)) = \h'(x,y) + (-1,0)  \quad \mbox{and}\quad \h' (Z' (x,y)) = \h'(x,y) + (0,1) ;
$$
\item $X'_{a,b} \cap X_{a',b'} = \emptyset$  for  $(a,b) \neq (a',b')$; and
\item $T'(X_{a,b})= X'_{a -1,b}$ and $Z'(X'_{a,b}) = X'_{a,b+1}$.
\end{enumerate}

Cut  $\widetilde M'$ along 
$$
\mathcal G' = \bigcup_{(a,b) \in \Z \times \Z} X'_{a,b}, 
$$
i.e., remove the locus $\mathcal G'$ from $\widetilde M'$ replacing each $X'_{a,b}$ with two copies $X_{a,b}^{\mbox{left}}$ and $X_{a,b}^{\mbox{right}}$ of
$X'_{a,b}$ intersecting only at the preimages of the endpoints of $\Gamma_{a,b}$ under ${\h'}^{-1}$, 
and let 
$\widetilde M^{\mbox{cut}}$ be the result. 
Let 
$$
\sG^{\mbox{cut}} = \bigcup_{(a,b) \in \Z \times \Z} X_{a,b}^{\mbox{left}} \cup X_{a,b}^{\mbox{right}}.
$$
Let $q' : \widetilde M^{\mbox{cut}} \rightarrow \widetilde M'$ be the quotient map identifying
both $X_{a,b}^{\mbox{left}}$ and $X_{a,b}^{\mbox{right}}$ with $X'_{a,b}$.

Let $T^{\mbox{cut}}, Z^{\mbox{cut}} : \widetilde M^{\mbox{cut}} \rightarrow \widetilde M^{\mbox{cut}}$
be the lifts of the maps $T'$ and $Z'$ so that
$$
T^{\mbox{cut}}(X_{a,b}^{\mbox{right}}) = X_{a-1,b}^{\mbox{right}}, \quad T^{\mbox{cut}}(X_{a,b}^{\mbox{left}}) = X_{a-1,b}^{\mbox{left}}, 
$$
and
$$
Z^{\mbox{cut}}(X_{a,b}^{\mbox{right}}) = X_{a,b+1}^{\mbox{right}}, \quad Z^{\mbox{cut}}(X_{a,b}^{\mbox{left}}) = X_{a,b+1}^{\mbox{left}}, 
$$
Figure~\ref{gluing-fig} gives a local illustration of the associated cuts along $\Gamma_{a,b}$ in $\R \times \R$, replacing each $\Gamma_{a,b}$
with a left and right copy attached along their boundaries.
\begin{figure}[htbp] 
   \centering
   \includegraphics[width=2.25in]{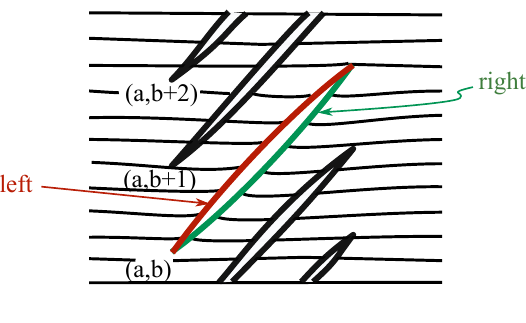} 
   \caption{A local illustration of the associated cuts along $\Gamma_{a,b}$ in $\R \times \R$.}
   \label{gluing-fig}
\end{figure}
Let $\sL^{\mbox{cut}}$ be the lift of $\sL'$ to $\widetilde M^{\mbox{cut}}$.  Then $\widetilde\sL^{\mbox{cut}}$ is a foliation of $\widetilde M^{\mbox{cut}}$ by
intervals whose boundaries lie on the left or right cut loci.

By the definitions, we have the following.

\begin{lemma} \label{cont-lem} The maps $T^{\mbox{cut}}$ and $Z^{\mbox{cut}}$ are commuting self-homeomorphisms of $\widetilde M^{\mbox{cut}}$
and generate a $\Z \times \Z$ action on $\widetilde M^{\mbox{cut}}$ that is properly discontinuous and free, and preserves the leaf structure of
$\widetilde\sL^{\mbox{cut}}$.
\end{lemma}

\subsubsection{Pasting}\label{pasting-sec}

We are now ready to define the new manifold $\widetilde M$ and foliation $\widetilde \sL$.   
For each $(a,b) \in \Z \times \Z$ we define an identification $X_{a,b}^{\mbox{left}}$ to $X_{a,b}^{\mbox{right}}$ as follows.
Let $\mbox{id}_{\mbox{\small r}} : X'_{a,b} \rightarrow X_{a,b}^{\mbox{right}}$ be the identification of $X'_{a,b}$ with $X_{a,b}^{\mbox{right}}$,
and $\mbox{id}_{\ell}$ be the identification of $X'_{a,b}$ with $X_{a,b}^{\mbox{left}}$.
Let 
$$
p_{a,b} : X_{a,b}^{\mbox{left}} \rightarrow X_{a,b}^{\mbox{right}}
$$
be the {\it pasting map} defined by the composition
$$
\xymatrix{
 X_{a,b}^{\mbox{left}}\ar[r]&X'_{a,b} \ar[r]^{\sigma'} \ar[l]_-{\mbox{id}_{\ell}}  &\sigma'(X_{a,b}) \ar[r]^-{\widetilde \zeta^b\widetilde \eta\widetilde \zeta^{-b}}
 &\sigma'(X_{a,b})  \ar[r]^-{(\sigma')^{-1}} &X'_{a,b}\ar[r]^-{\mbox{id}_{\mbox{\small r}}}&X_{a,b}^{\mbox{right}}\ar[l]
 }
$$
\begin{lemma} \label{homeo-lem}
For each $(a,b) \in \Z$, we have
 $$
 T^{\mbox{cut}} \circ p_{a,b} = p_{a-1,b} \circ T^{\mbox{cut}} \qquad  Z^{\mbox{cut}} \circ p_{a,b} = p_{a,b+1} \circ Z^{\mbox{cut}}.
 $$
 \end{lemma}
 
 Let $\widetilde M$ be obtained from $\widetilde M^{\mbox{cut}}$ by pasting each $X_{a,b}^{\mbox{left}}$ to $X_{a,b}^{\mbox{right}}$ by
 the map $p_{a,b}$, that is,  $x \in X_{a,b}^{\mbox{left}}$ is identified with $y \in X_{a,b}^{\mbox{right}}$ if
 $$
y =  p_{a,b}( x).
$$
Let 
 $$
 q : \widetilde M^{\mbox{cut}} \rightarrow \widetilde M
 $$
  be the quotient map,   
  $$
  X_{a,b} = q(X_{a,b}^{\mbox{left}}) (= q(X_{a,b}^{\mbox{right}})),
  $$
and  $\sG = q(\sG^{\mbox{cut}})$. Let 
  $$
  \tau : \widetilde M \rightarrow \widetilde M^{\mbox{cut}}
  $$
  be the (non-continuous) lifting map that is the inverse map of $q$ on $\widetilde M \setminus \sG$ and the inverse of the
  restriction $q : X_{a,b}^{\mbox{left}} \rightarrow q(X_{a,b}^{\mbox{left}})$ on $X_{a,b}$.

By Lemma~\ref{homeo-lem},  $T^{\mbox{cut}}$ and $Z^{\mbox{cut}}$ define homeomorphisms $T$ and $Z$ on $\widetilde M$ with the property that
  $$
 q \circ T^{\mbox{cut}} = T \circ q, \qquad q \circ Z^{\mbox{cut}} = Z \circ q.
 $$

Let $\widetilde \sL$ be the foliation on $\widetilde M$ defined by $\widetilde \sL^{\mbox{cut}}$ as follows.  
The inclusion
$$
\widetilde M' \setminus \sG' \subset \widetilde M^{\mbox{cut}}
$$
defines an identification
$$
\widetilde M' \setminus \sG'   = \widetilde M \setminus \sG.
$$
Thus $\sL'$ and $\widetilde \sL^{{\mbox {cut}}}$ define foliations on $ \widetilde M' \setminus \sG'$, and we have left to consider
only how the leaves of $\widetilde \sL^{\mbox{cut}}$ meet near a point on $q(X_{a,b}^{\mbox{left}})$.
If we think of the leaves of $\widetilde \sL^{\mbox{cut}}$ as being oriented left to right, then they meet each $X_{a,b}^{\mbox{left}}$ 
only at endpoints and meet each $X_{a,b}^{\mbox{right}}$ only at initial points.  
Furthermore, these intersections are locally
trivial products (as illustrated schematically in Figure~\ref{pasting-fig}).  
\begin{figure}[htbp] 
   \centering
   \includegraphics[width=1in]{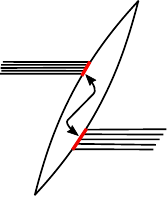} 
   \caption{A schematic picture (in one-dimensional lower) of how leaves of the foliation $\widetilde \sL^{\mbox{cut}}$ are pasted to form the leaves
   of $\widetilde \sL$..}
   \label{pasting-fig}
\end{figure}
In $\widetilde M$, neighborhoods of these endpoints and initial points are
identified under the map $q$ according to the pasting map $p_{a,b}$.  Since $p_{a,b}$ is a homeomorphism 
it follows that the resulting collection of leaves is a foliation. We have shown the following.

\begin{lemma} The leaf structure defined by $\widetilde \sL^{\mbox{cut}}$  on $\widetilde M^{\mbox{cut}}$ descends to a leaf structure for $\widetilde \sL$
on  $\widetilde M$.  
\end{lemma}

Since the $X_{a,b}$ are disjoint  we have the following.

\begin{lemma} \label{Haction-lem} The homeomorphisms $T$ and $Z$ commute and generate a free abelian group $H$ of rank $2$ acting freely
 and properly discontinuously on $\widetilde M$, and permuting the leaves of $\widetilde \sL$.
 \end{lemma}

\begin{remark}  {\em Following \cite{McMullen:Poly} we choose the notation $T$ to correspond to translation in $\R \times \R$ by
$(-1,0)$  in order to be compatible with the typical notation used for mapping tori:
\begin{eqnarray*}
M &=&  [0,1] \times S/_{\mbox{\small $(1,x) \simeq (0,\phi(x))$}}\\
&=&  \R \times S/_{\mbox{\small $(t,x) \simeq (t-1,\phi(x))$}}.
\end{eqnarray*}
 }
\end{remark}

\begin{remark} \label{leaves-rem}{\em Unlike the situation for $\widetilde \sL'$ in $\widetilde M'$ the corresponding
projection $\widetilde M \rightarrow \R$ 
need not be 
surjective on all leaves of $\widetilde \sL$ in $\widetilde M$.}\end{remark}

Let $\sigma : \widetilde M \rightarrow \widetilde S$ be the composition
$$
\xymatrix{
\widetilde M \ar[r]^-\tau &\widetilde M^{\mbox{cut}} \ar[r]^-{q'}& \widetilde M' \ar[r]^-{\sigma'} &\widetilde S.
}
$$
Let $\h : \widetilde M \rightarrow \R \times \R$  be the composition 
$$
\xymatrix{
\widetilde M \ar[r]^-\tau &\widetilde M^{\mbox{cut}} \ar[r]^-{q'}& \widetilde M' \ar[r]^-{\h'} &\R \times \R.
}
$$
Then both $\sigma$ and $\h$ have jump discontinuities on $\sG$,  but their restrictions to $\sG$ are continuous.
  In the next section we define fibrations of $\widetilde M$
over $\R$ whose fibers are homeomorphic to $\widetilde S$ by the restriction of $\sigma$, and so that each component
of $\sG$ is contained in a fiber.

\subsection{Proof of Theorem~\ref{quotient-thm}}\label{quotient-thm-sec}
Let $M = \widetilde M/H$ and let $\sL$ be the foliation defined by taking the image of the leaves of $\widetilde M$ in $M$.
 We will show that $M$ has a fibered face $F$ corresponding to cross-sections of $\sL$, and there is a linear embedding
$$
I_{m_0} \hookrightarrow F
$$
so that $I_{m_0}(\Q)$ maps to the rational points on the image of $I_{m_0}$, and for each $\r \in I_{m_0}(\Q)$,
the image of $\r$ in $F$ has monodromy equal to $(S_{\r},\phi_{\r})$.

Since $\widetilde M \rightarrow M$ is a regular abelian covering with automorphism 
group $H$, it is an intermediate covering of the maximal abelian covering of $M$,
and hence there is a corresponding epimorphism $H_1(M;\Z) \rightarrow H$.

Identifying $H^1(M;\R)$ with homomorphisms of $H_1(M;\Z)$ to $\R$, we obtain
an inclusion:
$$
\Hom(H;\R) \rightarrow H^1(M;\R).
$$
Let $W$ be the image of this map, and let $z,u \in W$ be the basis elements
 dual to  $Z$ and $T$. Using the notation $(r,s)$ for the linear combination $ru + sz$, 
so that 
$$
(r,s) \cdot (T^aZ^b) = ra + sb,
$$
let $V_Q \subset W$ be the cone defined by
$$
V_Q =\{(-r,-s) \in W \ | \ 0 < sm_0 < r\}.
$$

We prove Theorem~\ref{quotient-thm} by showing the following.

\begin{proposition}\label{monodromy-prop} There is a
fibered face $F \subset H^1(M;\R)$ and a continuous map
$$
\iota: I_{m_0} \rightarrow F \cap W
$$
with the property that for $\r \in I_{m_0}(\Q)$
$$
(S_{\r},\phi_{\r})=\mathfrak f(\iota(c)).
$$
\end{proposition}

\noindent
To prove Proposition~\ref{monodromy-prop} we show that for each $\r \in I_{m_0}(\Q)$, there is a 
corresponding element $\alpha_\r \in V_\Q$, and a
fibration $\widetilde \rho_\r : \widetilde M \rightarrow \R$ that lifts a fibration
of $\rho_{\r} : M \rightarrow S^1$ with monodromy equal to $(S_\r,\phi_\r)$.  We then show that the 
map $\r \mapsto \alpha_\r$ extends to a continuous one-to-one map from $I_{m_0}$ to $V_Q \cap F$ for
a fibered face $F$.
\medskip

\noindent 
{\bf Step 1. Building fibrations of $\widetilde M$ to $\R$ transverse to the foliation $\widetilde \sL$.}

\noindent
Take any $\r = \frac{k}{n} \in I_{m_0}(\Q)$, and let
$$
\alpha_{\r} = (-k,-n).
$$
Then we have $\alpha_{\r} \in V_\Q$, and, seeing $\alpha_{\r}$  as a homomorphism, we have
\begin{eqnarray*}
\alpha_{\r} : H &\rightarrow& \Z\\
T^aZ^b &\mapsto& - an - bk.
\end{eqnarray*}
Let 
$$
Z_{\r} = T^{-k}Z^n.
$$
Then the kernel $K_{\r}$ of $\alpha_{\r}$ is freely generated by $Z_{\r}$.
Let $w,\overline k_n$ be the solutions to 
$$
wn + \overline k_nk =  1,
$$
 and let
$$
T_{\r}= T^{w}Z^{\overline k_n}.
$$
Then $\alpha_{\r}(T_{\r}) = -1$.

We will define a fibration $\widetilde \rho_{\r} : \widetilde M \rightarrow \R$ so that the fibers of $\widetilde \rho_{\r}$
are homeomorphic by the restriction of $\sigma$ to $\widetilde S$, they are preserved by the elements of $K_{\r}$, and they are permuted by the actions of $Z$ and $T$.
Furthermore, each leaf of $\widetilde \sL$ intersects each fiber of $\widetilde \rho_{\r}$ exactly once.

We  begin by first defining a suitable projection $p_{\r} : \R \times \R \rightarrow \R$,
and then pulling back by $\h$ to $\widetilde M$.
\begin{figure}[htbp] 
   \centering
   \includegraphics[width=4in]{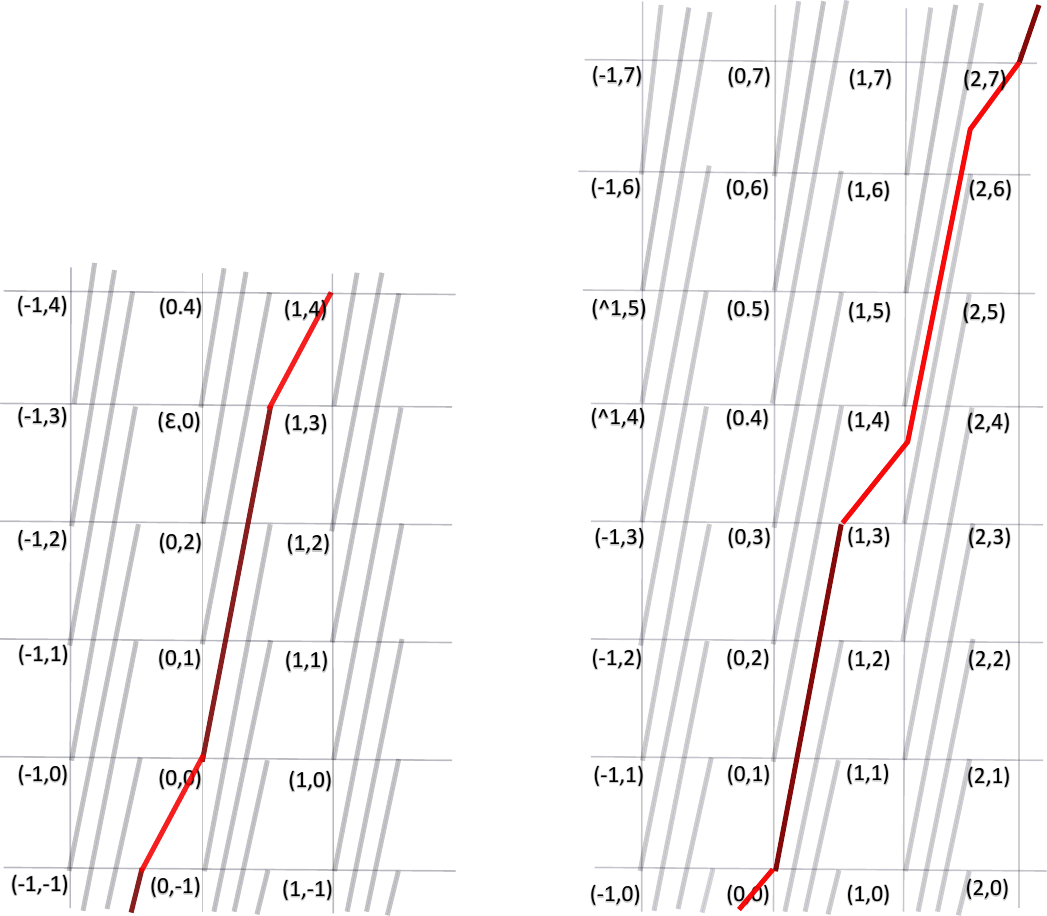} 
   \caption{The loci $\widetilde R_{\r} \subset \R \times \R$ for $\r= \frac{1}{4}$ (left) and $\r = \frac{2}{7}$ (right) when $m_0 = 3$.}
   \label{Gamma-fig}
\end{figure}

\begin{proposition}\label{R-prop} For each  $\r = \frac{k}{n} \in I_{m_0}(\Q)$ there is
a continuous monotone increasing function 
$$
g_{\r} : \R \rightarrow \R,
$$
such that
\begin{enumerate}
\item $g_{\r}(r) = \frac{r}{\r}$ for all $r \in \Z$; 
\item $g_{\r}(x + r) =  g_{\r}(x) + \frac{r}{\r}$ for all $r \in \Z$ and $x \in \R$; and
\item for each $(a,b) \in \Z \times \Z$,  $\Gamma_{a,b}$ is contained in the graph $y = g_{\r}(x) + \frac{r}{k}$ for some $r \in \Z$.
\end{enumerate}
\end{proposition}

\begin{proof}
Let $\Delta_{0,0}$ be the straight line segment on $\R \times \R$ connecting 
$p =(\frac{1}{2},m_0)$ and $q_{\r}=(1,\frac{1}{\r})$.
For $(a,b) \in \R \times \R$ let
$$
\Delta_{a,b} = (a,b) + \Delta_{0,0}.
$$
Let
$$
\widetilde R_{\r} = \bigcup_{r \in \Z} (\Gamma_{r,\frac{r}{\r}} \cup \Delta_{r,\frac{r}{\r}}).
$$
Then $\widetilde R_{\r}$ is the graph of a piecewise-linear, monotone increasing function $g_{\r}$
with the desired properties.
\end{proof}

Let 
$$
\widetilde R_{c,\xi} = \widetilde R_{\r} + (0,-\frac{\xi}{k})
$$
(see Figure~\ref{Gamma-fig}).   Then $\widetilde R_{c,\xi}$ intersects $\Gamma_{a,b}$ for some $(a,b) \in \Z \times \Z$ if and only if $\xi \in \Z$.
Each $\widetilde R_{c,\xi}$ is the graph of a monotone increasing,
piecewise linear, continuous function $g_{c,\xi}$ defined by $g_{\r,\xi} (\r) = g_{\r}(\r) + \xi$.   Thus, the $\widetilde R_{c,\xi}$
partition $\R \times \R$ into a disjoint union, and we have a well-defined continuous function:
$$
p_{\r} : \R \times \R \rightarrow \R
$$
defined by sending each $x \in \widetilde R_{c,\xi}$ to $\xi$.

\begin{lemma}\label{R-lem} For each $\xi \in \R$, the locus $\widetilde  R_{c,\xi}$ has the following properties.
\begin{enumerate}
\item translation by $(1,\frac{n}{k})$, generates an infinite cyclic action on $\widetilde R_{c,\xi}$ with
fundamental domain given by $(0,-\frac{\xi}k) + \Gamma_{0,0} \cup \Delta_{0,0}$;
\item $(0,1) + \widetilde R_{c,\xi} = \widetilde R_{c,\xi -k}$; and
\item 
$(-1,0) + \widetilde R_{c,\xi} = \widetilde R_{c,\xi  - n}$.
\end{enumerate}
\end{lemma}

\begin{proof}
To prove (1) we note that the statement holds for $\widetilde R_{\r}$ by construction.  Since translations by $(0,\xi)$ and by $(1,\frac{n}{k})$ commute
on $\R \times \R$, the statement also holds for $\widetilde R_{c,\xi}$ for all $\xi$.  
Properties (2) and (3) follow from verifying that $(0,1) \in \widetilde R_{c,-k}$ and
$(0,\frac{n}{k}) \in \widetilde R_{c,-n}$.
\end{proof}

We now apply Lemma~\ref{R-lem} to $\widetilde M$.
Let $\widetilde \rho_{\r} = p_{\r} \circ \h$, and let 
$ \widetilde S_{c,\xi} =  \h^{-1}(\widetilde R_{c,\xi}) \subset \widetilde M$.    

\begin{lemma} \label{coveringfibration-lem} The map $\widetilde \rho_{\r}$ is a fibration with fibers $\widetilde S_{c,\xi}$, for all $\xi \in \R$
and for all $\xi \in \R$ we have the following:
\begin{enumerate}
\item the fundamental domain of the action of $K_{\r}$ on $\widetilde S_{c,\xi}$ is the set
$$
\h^{-1}\left (\bigcup_{r = 0}^{k-1}\Gamma_{r,\frac{rn}{k}} \cup \Delta_{r,\frac{rn}{k}} \right ).
$$
\item $K_{\r} = \langle Z_{\r} \rangle$ generates the set-wise stabilizer in $H$ of $\widetilde S_{c,\xi}$; 
\item $Z(\widetilde S_{c,\xi}) = \widetilde S_{c,\xi -k }$ and
$T(\widetilde S_{c,\xi}) = \widetilde S_{c,\xi -n}$; 
\item $Z_{\r}(\widetilde S_{c,\xi}) = \widetilde S_{c,\xi}$, $T_{\r}(\widetilde S_{c,\xi}) = \widetilde S_{c,\xi -1}$; 
\item \label{Tc-item}$\widetilde \rho_{\r} (Z_{\r}(\widetilde S_{c,\xi})) = \widetilde \rho_{\r}(\widetilde S_{c,\xi})$, $\widetilde \rho_{\r}(T_{\r}(\widetilde S_{c,\xi})) = \widetilde \rho_{\r}(\widetilde S_{c,\xi}) -1$; and
\item $\widetilde \rho_{\r}$ restricts to a homeomorphism to $\R$ on each leaf of $\widetilde \sL$.
\end{enumerate}
\end{lemma}

\begin{proof}
The map $\widetilde \rho_{\r}$ is continuous since each of the fibers $\widetilde S_{c,\xi}$ is either disjoint from $\sG$ or contains components of $\sG$.
Since  $\h \circ Z_{\r} \circ \h^{-1}$
is translation by $(k,n) = k (1,\frac{n}{k})$,  Lemma~\ref{R-lem} implies that $Z_{\r}$ and all its powers preserve $\widetilde S_{\r,\xi}$. 
To show that $Z_{\r}$ generates the stabilizer of $\widetilde S_{c,\xi}$ we note that any element of $H$ is conjugate to a translation of
$\R\times \R$.   If $h \in H$ preserves $\widetilde S_{\r,\xi}$ it must be conjugate to 
a translation preserving $\widetilde R_{\r,\xi}$, i.e. translation by $m(1,\frac{n}{k})$ for some integer $m$.  Since $H$
furthermore preserve $\sG$ and $k$ and $n$ are relatively prime we must also have $k$ divides $m$.  Thus, $h$ is a multiple of $Z_{\r}$.
This proves (1) and (2).

Items (3) and (4) follow directly from Lemma~\ref{R-lem}.
To prove (5) we note that
the action of $T_{\r}$ on $\widetilde M$ is conjugate by $\h$ to translation in $\R \times \R$ by
$$
T_{\r} : (a,b) \mapsto (a,b) + (-w,\overline k_n).
$$
Since translation by $(1,\frac{n}{k})$ stabilizes each 
$\widetilde R_{c,\xi}$,  $(-w,\xi+\overline k_n)$ and $(0,(\xi+\overline k_n) + \frac{wn}{k})$ lie
on the same fiber of $p_{\r}$.  Using the definition of $w$ and $\overline k_n$, we have
$$
(\xi + \overline k_n) + \frac{wn}{k} = \xi + \frac{\overline k_n k + wn}{k} = \xi + \frac{1}{k}
$$
proving (\ref{Tc-item}). 
To prove (6) it suffices to show that $\widetilde \rho_{\r}$ is 1-1 and onto on each leaf $\ell$ of $\widetilde \sL$,
or equivalently that $\ell$ intersects $\widetilde S_{c,\xi}$ exactly once for each $\xi \in \R$.
This can be seen by verifying that
each segment of $\widetilde \sL'$ intersects each $(q' \circ \tau)(\widetilde S_{c,\xi})$ in exactly one point.
\end{proof}

By the above we have a setwise bijection
\begin{eqnarray*}
\mathfrak j_{\r} : \widetilde M &\rightarrow& \R \times \widetilde S\\
s &\mapsto& (\widetilde \rho_{\r}(s), \sigma(s))
\end{eqnarray*}
which is a homeomorphism outside the preimage of $\Z \times \widetilde S$.
Outside this locus of discontinuity, the definition of the pasting map used to construct $\widetilde M$ gives:
$$
 \mathfrak j_{\r} \circ T_{\r} \circ \mathfrak j_{\r}^{-1}(s) = (\widetilde \rho_{\r}(s) - 1, \widetilde \phi_{\r}^{-1}(\sigma(s))).
 $$
 Thus, we can also think of $\mathfrak j_{\r}$ as a homeomorphism after cutting $\R \times \widetilde S$ along $\Z \times \widetilde S$ and
pasting by the map $\widetilde \phi_{\r}$.

Define a forward flow along $\widetilde \sL$ by
\begin{eqnarray*}
f_{\r} : \R \times \widetilde M &\rightarrow& \widetilde M\\
(t,s) &\mapsto& \mathfrak j_{\r}^{-1}(\widetilde \rho_{\r}(s) + t, \phi_{\r}^{r}(\sigma(s)))
\end{eqnarray*}
where 
$$
r = \left \{
\begin{array}{ll}
\lfloor t \rfloor & \quad \mbox{if $t \ge 0$}\\
\lceil t \rceil & \quad \mbox {if $t < 0$}
\end{array}
\right .
$$
Another way to think of $f_{\r}$ is that for each $(t,s) \in \R \times \widetilde M$, $f_{\r}(t,s)$ is the point on $\widetilde S_{\r,\xi+t}$
that lies on the leaf $\ell$, where $\ell$ is the leaf containing the point $s$ and $\xi = \widetilde \rho_{\r}(s)$.
\medskip
     
\noindent    
{\bf Step 2. Comparing $T_{\r}$ with the forward flow on $\widetilde M$.  }

\begin{lemma} \label{flowforward-lem} The flow map $f_{\r}$ satisfies
$$
\sigma(f_{\r}(t+1,s)) = \widetilde \phi_{\r}(\sigma(f_{\r}(t,s))).
$$
\end{lemma}

\begin{proof}  This follows from comparing the definition of the pasting map in Section~\ref{pasting-sec}
with the definition of $\widetilde \phi_{\r}$ given in Section~\ref{quot-sec}.\end{proof}

\begin{corollary}\label{monodromy-cor}
The  map $T_{\r}$ satisfies
$$
\sigma(s) = \widetilde \phi_{\r} (\sigma(T_{\r}(s)))
$$
for $s \in \widetilde M$.
\end{corollary}

\begin{proof}  By Lemma~\ref{flowforward-lem} we have
$$
\widetilde \phi_{\r}(\sigma(T_{\r}(s))) = \widetilde \phi_{\r} ( \sigma(f_{\r}(-1,s))) =\sigma(f_{\r}(0,s)) = \sigma(s).
$$
\end{proof}

\noindent
{\bf Step 3. Descending to the quotient fibration.}
Since  $Z_{\r}$ and $T_{\r}$
commute, it follows that $T_{\r}$ defines a covering automorphism $\overline T_{\r}$ on the intermediate covering
$$
M_{\r} = \widetilde M/\langle Z_{\r} \rangle \rightarrow M  
$$
over $M$.  
 The projection $\widetilde \rho_{\r}$ also descends to a
fibration
$$
\overline \rho_{\r} : M_{\r} \rightarrow \R
$$
whose fibers are homeomorphic to $S_{\r} = \widetilde S/\widetilde \zeta^n$.   For $s \in \widetilde M$, let $\overline s$ be the image of $s$ in $M_{\r}$, and
let $\overline{\sigma(s)}$ be the image of $\sigma(s)$ in $S_{\r}$.

Since $Z_{\r}$ is conjugate by $\j_{\r}$ to  $\mbox{\small id} \times \zeta^n$ on $\R \times \widetilde S$, 
the map $\j_{\r}$ descends to an
identification
\begin{eqnarray*}
\overline j_{\r} : M_{\r} &\rightarrow& \R \times S_{\r}\\
\overline s &\mapsto& (\widetilde \rho_\r(s),\sigma_\r(s)).
\end{eqnarray*}

Then the element $\alpha_{\r}$ defines a commutative diagram
$$
\xymatrix{
\widetilde M\ar[d]_{/_{K_{\r}}}\ar[r]^-{\j_{\r}} &\R \times \widetilde S \ar[d]^{/_{\mbox{\tiny id} \times \widetilde \zeta^n}}\\
M_{\r} \ar[d]_{/_{\langle \overline T_{\r} \rangle}}\ar[r]^-{\overline \j_{\r}} &\R \times S_{\r} \ar[d]^{/_{\Z \times S_{\r}}}\\
M \ar[r] &S^1.
}
$$
where induced map $\overline T_{\r} : M_{\r} \rightarrow M_{\r}$ is defined by
\begin{eqnarray*}
\overline j_{\r} \circ \overline T_{\r} \circ j_{\r}^{-1} : \R \times S_{\r} &\rightarrow& \R \times S_{\r}\\
(t,x) &\mapsto& (t-1, \phi_{\r}(x))
\end{eqnarray*}

The bottom horizontal map in the diagram is the fibration $\rho_{\r}$ associated to $\r$, and the fibers are 
homeomorphic to $S_{\r}$ and are cross-sections of the flow $\sL$.  The monodromy is defined by $\overline T_{\r}$,
and equals $(S_{\r},\phi_{\r})$.  
This completes the proof of Theorem~\ref{quotient-thm}.

\subsection{Stable quotient families}\label{stable-sec}  
To prove Theorem~\ref{stability-thm}, we study the cohomology class in $H^1(M;\Z)$ associated to $\alpha_0 = (-1,0)$.
This is the primitive element on the limit of the sequence of rays defined by $(-k,-n)$ as $\frac{k}{n}$ approaches $0$.
The kernel of $\alpha_0$ restricted to $H$ is generated by the map $Z$.   For $r \in \Z$, let  $\widetilde S_{0,r}$ be the surface defined
by the image in $\widetilde M$ of $\{r\} \times \widetilde S \subset \widetilde M^{\mbox{cut}}$.   
 Unlike in the case of $\r \in I_{m_0}(\Q)$, there is no easily defined projection $\widetilde \rho_0$.
 
We will show that if $Q$ is a stable family, then for any $r \in \Z$ 
each leaf of $\widetilde \sL$ passes through  $\widetilde S_{0,r}$ exactly once for all $r \in \Z$, and
if $Q$ is not stable, then for each $r \in \Z$ there is at least one leaf of $\widetilde \sL$ that does not pass through $\widetilde S_{0,r}$.

 \begin{figure}[htbp] 
   \centering
   \includegraphics[height=2.5in]{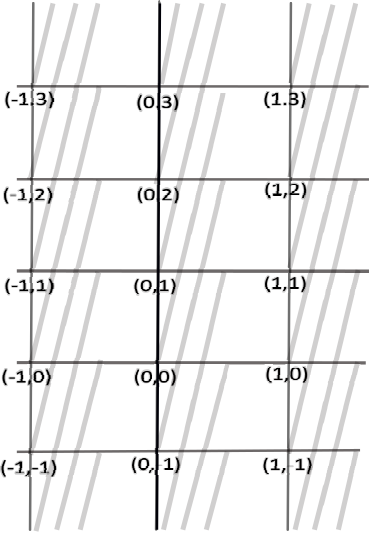} 
   \caption{The surfaces $\widetilde S_{0,r}$ are sent by $\h$ to the vertical lines $\{r\} \times \R$ where $r \in \Z$.}
   \label{Gammalimit-fig}
\end{figure}

\begin{lemma}\label{widetildephi-lem} Let $x \in \widetilde S$, and $r_1 < r_2$ be integers such that $\zeta^{-r_i}(x) \in \Sigma_0$ for $i=1,2$.
Then
we have
$$
 \widetilde \zeta^{-m_1 + r_1}(\widetilde \zeta \widetilde \eta)^{m_1} \widetilde \zeta^{-r_1}(x)
 =  \widetilde \zeta^{-m_1 + r_2}(\widetilde \zeta \widetilde \eta)^{m_1} \widetilde \zeta^{-r_2}(x).
 $$
 \end{lemma}
 \begin{proof}
 By assumption, $\Sigma_0 \cap \widetilde \zeta^2(\Sigma_0)  = \emptyset$, and hence we have $r_2 = r_1 + 1$.
 This means that $\widetilde \zeta^{-r_2} (x) \in \Sigma_0 \cap \widetilde \zeta^{-1}(\Sigma_0)$ and hence
 $$
 \widetilde \eta \widetilde \zeta^{-r_2}(x) = \zeta^{-r_2}(x).
 $$
 By the definition of stability we also have
 $$
 \widetilde \eta (\widetilde \zeta \widetilde \eta)^{m_1} \widetilde \zeta^{-r_2}(x) = (\widetilde \zeta \widetilde \eta)^{m_2} \widetilde \zeta^{-r_1}(x).
 $$
Putting this together gives the equalities:
\begin{eqnarray*}
 \widetilde \zeta^{-m_1 + r_1}(\widetilde \zeta \widetilde \eta)^{m_1} \widetilde \zeta^{-r_1}(x)
 &=& \widetilde \zeta^{-m_1 + r_2 -1}(\widetilde \zeta \widetilde \eta)^{m_1}\widetilde \zeta^{-r_2+1}(x)\\
 &=&\widetilde \zeta^{-m_1 + r_2} \widetilde \zeta^{-1} (\widetilde \zeta \widetilde \eta)^{m_1} \widetilde \zeta 
 \widetilde \zeta^{-r_2}(x)\\
 &=&\widetilde \zeta^{-m_1 + r_2} \widetilde \eta (\widetilde \zeta \widetilde \eta)^{m_1-1}\widetilde \zeta \widetilde \eta
 \widetilde \zeta^{-r_2}(x)\\
  &=&\widetilde \zeta^{-m_1 + r_2} (\widetilde \zeta \widetilde \eta)^{m_1} \widetilde \zeta^{-r_2}(x)\\
\end{eqnarray*}
completing the proof.
\end{proof}

Let $\widetilde \phi$ be defined by
\begin{eqnarray*}
\widetilde \phi : \widetilde S &\rightarrow& \widetilde S\\
x& \mapsto& \widetilde \zeta^{-m_1 + r}(\widetilde \zeta \widetilde \eta)^{m_1} \widetilde \zeta^{-r}(x),
\end{eqnarray*}
where $r$ is any  integer such that $\widetilde \zeta^{-r}(x) \in \Sigma_0$.  
This is well-defined by Lemma~\ref{widetildephi-lem}.

\begin{lemma} \label{zcommute-lem}  The map $\widetilde \phi$ commutes with $\widetilde \zeta$.
\end{lemma}

\begin{proof}  Let $x \in \widetilde \zeta^r(\Sigma_0)$.  Then we have
\begin{eqnarray*}
\widetilde \phi (\widetilde \zeta(x)) &=& \widetilde \zeta^{-m_1+r+1} (\widetilde \zeta\widetilde \eta)^{m_1}\widetilde \zeta^{-r-1}\widetilde \zeta(x)\\
&=&\widetilde \zeta \widetilde \zeta^{-m_1+r} (\widetilde \zeta\widetilde \eta)^{m_1}\widetilde \zeta^{-r}(x)\\
&=& \widetilde \zeta \widetilde \phi(x).
\end{eqnarray*}
\end{proof}

\begin{lemma} \label{welldefinedhomeo-lem} 
If $Q$ is stable then the map $\widetilde \phi$ is a continuous, locally injective, open mapping, and if $Q$ is bi-stable, then $\widetilde \phi$
is a homeomorphism.
\end{lemma}
\begin{proof} 
Since
\begin{eqnarray*}
\widetilde \zeta^{-m_1 + r}(\widetilde \zeta \widetilde \eta)^{m_1} \widetilde \zeta^{-r}(x) &=& \widetilde \zeta^r (\widetilde \zeta^{-m_1}\widetilde \eta \widetilde \zeta^{m_1}) \cdots (\widetilde \zeta^{-1} \widetilde  \eta \zeta) \widetilde \eta \widetilde \zeta^{-r}(x)
\end{eqnarray*}
we can think of $\widetilde \phi$ as the infinite composition of maps
$$
\psi_i = \widetilde \zeta^{-i} \widetilde \eta \widetilde \zeta^i
$$
where $i$ goes from $- \infty$ to $\infty$;
that is, 
$$
\widetilde \phi  = \cdots \circ \psi_{i-1} \circ \psi_i \circ \psi_{i+1} \circ \cdots,
$$
This is well-defined and continuous since for each $x \in \widetilde S$, if we let $r$ be such that $x \in \zeta^r(\Sigma_0)$,
we have 
$$
\psi_i(x) = x
$$
for all $i > -r$, and 
\begin{eqnarray*}
\widetilde \phi(x) &=&  \cdots \psi_{-r-1} \psi_{-r}(x) \\
&=&\psi_{-r-m_1} \cdots \psi_{-r}(x).
\end{eqnarray*}
Since $\psi_i$ are all homeomorphisms, this implies that $\widetilde \phi$ is an open mapping.

If in addition, $Q$ is bi-stable, then a similar argument shows that the following map is well-defined and continuous.
Take $x \in \widetilde S$, and let $r$ be such that $x \in \widetilde \zeta^{r}(\Sigma_{m_0})$.  Define
$$
\widetilde \phi' (x) = \widetilde \zeta^r (\widetilde \eta^{-1}\widetilde \zeta^{-1})^{m_1}\widetilde \zeta^{-r}(x).
$$
Then the stability of $Q(\widetilde S,\widetilde \zeta^{-1},\widetilde \eta^{-1})$ implies that for $x \in \widetilde S$
and $i$ large enough $\psi_i^{-1}(x) = (x)$, and  the composition
$$
\cdots  \circ \psi_{i+1}^{-1} \circ \psi_i^{-1} \circ \psi_{i-1}^{-1} \circ \cdots,
$$
is well-defined and continuous.   This map is the inverse of $\widetilde \phi$.
\end{proof}

\begin{lemma}\label{commutingZ-lem}  The map $\widetilde \phi$ commutes with $\zeta$.  
\end{lemma}

\begin{proof}  Take any $s \in \widetilde S$, and assume that $\widetilde \zeta^{-r}(s) \in \Sigma_0$.  Then
\begin{eqnarray*}
\widetilde \phi(\widetilde \zeta(s)) &=& \widetilde \zeta^{r+1} (\widetilde \zeta\widetilde  \eta)^{m_1}\widetilde \zeta^{-r-1}(\zeta(s))\\
&=& \widetilde \zeta^{r+1} (\widetilde \zeta\widetilde  \eta)^{m_1}\widetilde \zeta^{-r}(s)\\
&=& \widetilde \zeta \widetilde \zeta^{r} (\widetilde \zeta\widetilde  \eta)^{m_1}\widetilde \zeta^{-r}(s)\\
&=& \widetilde \zeta (\widetilde \phi(s)).
\end{eqnarray*}
\end{proof}

\begin{lemma}\label{leaf-lem} The quotient family $Q$ is bi-stable if and only if $a \in \Z$ and any leaf $\ell$ of
$\widetilde \sL$, $\ell$ intersects $\widetilde S_{0,a}$ in a single point $x_a$.  In this case, $x_a$ satisfies
 $$
\sigma(x_a) = \widetilde \phi(\sigma(x_{a+1}))
$$
for all $a \in \Z$.
\end{lemma}

\begin{proof} 
Let  
$$
\widetilde \rho_0: \widetilde M \rightarrow \R
$$
be the composition of $\h$ with projection of $\R \times \R \rightarrow \R$ onto the first coordinate.  
Then $\widetilde \rho_0$ has jump discontinuities at the cut locus $\sG$.  

Assume $Q$ is bi-stable.
Let $\ell$ be any leaf of $\widetilde \sL$.	   Then for each $a \in \Z$, $\ell$ intersects $\widetilde S_{0,a}$ in at most one point
since the restriction of  $\widetilde \rho_0$ to $\ell$ is monotone in small enough neighborhoods of
$$
\bigcup_{a \in \Z} \widetilde S_{0,a}.
$$

For each $a \in \Z$, let $\sG_a = \bigcup_{b \in \Z} X_{a,b}$.

Suppose $t_0$ is a point on the intersection $X_{a,b} \cap \ell$ for some $a,b \in \Z \times \Z$.
 Then there is a well-defined orientation on $\ell$ such that the next time $\ell$
intersects $\sG$ is at $X_{a,b-1}$.  This orientation defines an ordering on the sequence $ t_0, t_1,t_2,\dots$ of
intersections of $\ell$ with $\sG$ that occur on $\ell$ after $t_0$.
Then $t_i \in X_{a,b -i}$ for each $i = 1,2,\dots$, and
\begin{eqnarray*}
\sigma(t_i) &=& \widetilde\zeta^{b-i+1}\widetilde \eta\widetilde  \zeta^{-b + i-1}(\sigma(t_{i-1})) \\
&=& \widetilde\zeta^{b-i+1}\widetilde \eta \widetilde\zeta^{-b+i-1} \widetilde\zeta^{b-i +2} \widetilde\eta \widetilde\zeta^{-b+i-2}(\sigma(t_{i-2})) =
\widetilde \zeta^{b-i}(\widetilde\zeta\widetilde\eta)^2 \widetilde\zeta^{-b+i-2}(\sigma(t_{i-2}))\\
&=& \widetilde \zeta^b (\widetilde \zeta \widetilde \eta)^{b} \widetilde \zeta^{-b}(t_0).
\end{eqnarray*}
If $Q$ is stable, then this sequence must terminate after at most $m_1$ steps.  Since this is true for any starting point $t_0 \in \ell \cap \sG_a$,
it follows that $\ell \cap \sG_a$ is finite and has at most $m_1$ elements.
Furthermore, if $t_0$ is the first time $\ell$ meets $\sG_a$ and $t_m$ is the last, then
$$
\sigma(x_{a+1}) = \widetilde \zeta^{m+1} \widetilde \eta  \widetilde \zeta^{1-m} \sigma(t_m)  = \sigma(t_m) = \widetilde \zeta^r(\widetilde \zeta \widetilde \eta)^r \widetilde \zeta^{-4}(\sigma(t_0)) = \widetilde \phi(\sigma(t_0)) = \widetilde \phi(\sigma(x_a)
$$
where $x_a$ is the intersection of $\ell$ with $\widetilde S_{0,a}$ and $x_{a+ 1}$ is the intersection of $\ell$ with $\widetilde S_{0,a+1}$.

Since all leaves of $\widetilde \sL$ must interesect $\sG$ in at least one point (since $m_0 \ge 1$), we have shown that
for each $\ell \in \widetilde \sL$, $\ell$ intersects $\widetilde S_{0,a}$ and $\widetilde S_{0,a+1}$ for some $a$.  Our argument also
shows by induction that this $\ell$ intersects $\widetilde S_{0,a'}$ for all $a' \ge a$.  

Using the same argument on the backward flow on $\ell$ starting from the point $x_a$, the
stability of $Q(\widetilde S,\widetilde \zeta^{-1},\widetilde \eta^{-1})$
implies that $\ell$ intersects every $\widetilde S_{0,a'}$ for $a'<a$.

Conversely, suppose that for each $\ell$ and each $a \in \Z$, we have $\ell$ intersecting $\widetilde S_a$ in a single point $x_a$.
Again, we consider the ordered intersections of $\ell$ with $\sG_a$.  This must be a finite sequence, since if not, then $\ell$ would never reach
$\widetilde S_{a+1}$.  If $t_0,\dots,t_k$ are intersections, starting with the one on $\ell$ immediately $x_a$, then as before we have
$$
\widetilde \phi(\sigma(x_a)) = \widetilde \phi(\sigma(t_0)) = \sigma(t_k) = \sigma(x_{a+1}).
$$
\end{proof}

\begin{remark} {\em We have shown in the proof of Lemma~\ref{leaf-lem}, that $Q$ is stable if and only if
for all $a \in \Z$, and leaf $\ell \in \widetilde \sL$ passing through a point $x_a \in \widetilde S_a$,  the leaf
$\ell$ passes through $\widetilde S_{a'}$ for all $a' \ge a$.  A priori, there may be other leaves of $\ell$ that
stay in between $S_a$ and $S_{a+1}$.
}
\end{remark}

\noindent
By Lemma~\ref{leaf-lem}, if $Q$ is bi-stable then $\iota$ extends continuously to $0$, and
$\iota(0)$ has monodromy $(S,\phi)$, where $S = \widetilde S/\wzeta$ and $\phi$ is the map induced by $\widetilde \phi$.
Thus, $\alpha_0$ must lie in the interior of the fibered face and
by continuity of $L$, we have 
$$
\lim_{c \rightarrow 0} L(\alpha_{\r}) = L(\alpha_0) = L(S,\phi).
$$
Conversely, if $Q$ is not stable, then 
$\alpha_0$ must lie on the boundary of the fibered cone and the sequence $ L(\alpha_{\r}) $
diverges as $\r$ approaches $0$.
 This completes the proof of
Theorem~\ref{stability-thm}.

\section{Penner example}\label{example-sec}

In this section we illustrate Theorem~\ref{quotient-thm} and Theorem~\ref{stability-thm} using 
Penner's sequence $(R_g,\psi_g)$  (recall Figure~\ref{Pennerexample-fig}).

Let $Q= Q(\widetilde S,  \widetilde \zeta, \widetilde \eta)$ where
$\widetilde S$ is the 
infinite surface drawn in Figure~\ref{original_lifts-fig} as a stack of copies $\Sigma_i$,     $i \in \Z$, of a fundamental domain
$\Sigma$, and $\widetilde \zeta$ sends each $\Sigma_i$ homeomorphically to $\Sigma_{i+1}$, $\Sigma_i \cap \Sigma_{i+1} \subset \partial \Sigma_i$,
$\widetilde \zeta^2(\Sigma_i) \cap \Sigma_i = \emptyset$
and $\widetilde \eta = \delta_{\widetilde c}\delta_{\widetilde b}^{-1} \delta_{\widetilde a}$.

For a mapping class $(S,\phi)$, where $S$ has punctures or boundary components, let $\overline S$ be the {\it closure} of $S$, that is, the
closed surface obtained by filling in the punctures and boundary components, and let  $(\overline S,\overline \phi)$ 
be the induced mapping class.
Then we observe the following.

\begin{proposition} For $g \ge 3$, 
$$
(R_g,\psi_g) = (\overline S_{\frac{1}{g}},\overline \phi_{\frac{1}{g}}), 
$$
where $(S_{\frac{1}{g}}, \phi_{\frac{1}{g}}) = \iota(\frac{1}{g}) \in Q$.
\end{proposition}

\begin{proof}  It suffices to observe that $\phi_{\frac{1}{g}} = \zeta_g \delta_c \delta_b^{-1}\delta_a$.
\end{proof}

\begin{corollary}  For all $\r \in (0,\frac{1}{2})$, $(S_{\r},\phi_{\r})$ is pseudo-Anosov.
\end{corollary}

We now use Theorem~\ref{stability-thm} to show that $L(S_{\frac{1}{g}},\phi_{\frac{1}{g}})$ is bounded for $g \ge 2$.  By Theorem~\ref{stability-thm}
it suffices to show the following.

\begin{proposition}   The family  $Q$ is a bi-stable quotient family.
\end{proposition}
\begin{proof}
The support of $\widetilde \eta$ is contained in the union of annular neighborhood of $\widetilde a, \widetilde b, \widetilde c$, which are
all contained in 
$$
\Sigma_0 \cup \Sigma_1
$$
and hence $m_0=1$
Thus, we have
$$
\widetilde \eta (\Sigma_0) \subset \Sigma_0 \cup \Sigma_1,
$$
and
$$
\widetilde \zeta \widetilde \eta (\Sigma_0 \cup \Sigma_1) \subset \Sigma_1 \cup \Sigma_2.
$$
On $\Sigma_1$, $\delta_{\widetilde a}$ and $\delta_{\widetilde b}$ act trivially, while the
map $\delta_{\widetilde c}$ sends $\Sigma_1 \cup \Sigma_2$ to the union of $\Sigma_1 \cup \Sigma_2$ and a small
annular neighborhood of $\widetilde c$.   Thus 
$$
\widetilde \zeta \widetilde \eta(\Sigma_1 \cup \Sigma_2) \subset \Sigma_2 \cup \Sigma_3 \cup \widetilde \zeta(A_{\widetilde c})
$$
where $A_{\widetilde c} \subset \Sigma_0 \cup \Sigma_1$ is a small  annular neighborhood of $\widetilde c$.
The map $\widetilde \eta$ acts trivially on $\Sigma_2 \cup \Sigma_3 \cup \widetilde \zeta(A_{\widetilde c})$. 
Thus, we have, for all $x \in \Sigma_0$,
$$
(\widetilde \zeta \widetilde \eta)^3(x) = \widetilde \zeta (\widetilde \zeta \widetilde \eta)^2(x).
$$

Now consider $Q(\widetilde S,\widetilde \zeta^{-1},\widetilde \eta^{-1})$.
The support of 
$$
\eta^{-1} = \delta_{\widetilde a} \delta_{\widetilde b}^{-1} \delta_{\widetilde c}
$$
is $\Sigma_0 \cup \Sigma_1$ and its image is the union of annuli around $\widetilde a$, $\widetilde b$ and $\widetilde c$.
Thus
$$
\widetilde \zeta^{-1}\widetilde \eta^{-1} (\Sigma_0 \cup \Sigma_1) \subset \Sigma_{-1} \cup \widetilde \zeta^{-1}(A_{\widetilde c})
$$
The action of $\delta_{\widetilde c}$ acts trivially on this set, and hence
$$
(\widetilde \zeta^{-1}\widetilde \eta^{-1})^2(\Sigma_0 \cup \Sigma_1) \subset \Sigma_{-2} \cup \Sigma_{-1}.
$$
The map $\widetilde \eta^{-1}$ acts trivially on this set, and hence $Q$ is bi-stable.
\end{proof}

By Theorem~\ref{stability-thm} it follows that
$$
\lim_{g \rightarrow \infty} L(S_{\frac{1}{g}},\phi_{\frac{1}{g}}) =  L(S_0,\phi_0)
$$
where  $S_0 = \widetilde S/\widetilde \zeta$ and $\phi_0$ is defined by the $\zeta$-equivariant map $\widetilde \phi$.
This gives an alternative to Penner's proof in  \cite{Penner91}  that for some constant $C > 0$ 
$$
\log(\lambda(\phi_{\frac{1}{g}})) \leq \frac{C}{g}
$$
for $g \ge 2$.

\subsection{Limiting mapping class}
We describe the limiting mapping class $(S,\phi)$ for $Q$ at $0$  explicitly.  Figure~\ref{Pennerwedge-fig} gives a picture of 
$S = \widetilde S/\zeta$ with images $a,b$, $c$ of the curves $\widetilde a,\widetilde b$ and $\widetilde c$, and
the image $d$ of $\widetilde d = \Sigma_0 \cap \zeta(\Sigma_0)$.
Then, since the map $\widetilde \zeta$ descends to the identity map, $\phi$ is the mapping class on the torus with two 
boundary components given by the composition
$\phi = \delta_{\r} \circ \delta_b^{-1} \circ \delta_a$.

\begin{figure}[htbp]
\begin{center}
\includegraphics[height=1.5in]{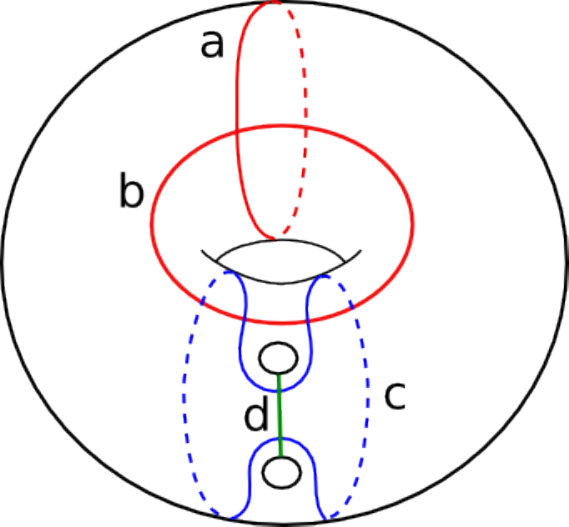}
\caption{The limiting mapping class for Penner's sequence.}
\label{Pennerwedge-fig}
\end{center}
\end{figure}

\subsection{Alexander and Teichm\"uller polynomial}
Let $M$ be the mapping torus of the quotient family $Q$.

\begin{proposition} The first Betti number of $M$ equals $2$.
\end{proposition}

\begin{proof}
The first cohomology group of $H^1(S;\Z)$ is generated by duals to $[a]$, $[b]$ and $[d]$, the 
relative homology classes defined by $a,b$ and $d$ in $H_1(S,\partial S;\Z)$.
With respect to this basis, the action of $\phi$ on the first cohomology group $H^1(S,\Z)$  is given by
$$
\left [
\begin{array}{ccc}
1 & 1 & 0\\
1 & 2 & 0\\
0 & 0 & 1
\end{array}
\right ]
$$
The  invariant cohomology is 1-dimensional, and hence
$b_1(M) = 2$.   
\end{proof}

The cohomology class generating the invariant cohomology is dual to the path $d$ between the two punctures on $S$, and
the corresponding  cyclic covering $\widetilde S \rightarrow S_0$ is the one drawn in
 Figure~\ref{original_lifts-fig}, with fundamental domain $\Sigma = S\setminus [d]$  the surface $S$ slit at
 $d$.  Let $\widetilde \zeta$ generate the group
of covering automorphisms of $\R \times \widetilde S \rightarrow M$.  Then
$Z = \widetilde \zeta\times\{\mbox{id}\}$ and $T = T_{\widetilde \phi}$ define covering automorphisms, and hence
generators for $H_1(M;\Z)$.    
Let $u,t \in H^1(M;\Z)$ be duals to $Z$ and $T$ respectively.  

Let $\tt$ be the train track for $\phi_0$ given by smoothings near intersections of the union of $a,b$ and $c$ as in 
Figure~\ref{Pennertt-fig}.  
Endow $a,b$ and $c$ with orientations.  
\begin{figure}[htbp] 
   \centering
   \includegraphics[width=3in]{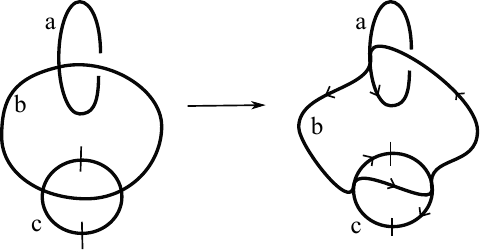} 
   \caption{Traintrack  $\tt$ for the quotient mapping class $(S_0,\phi_0)$.}
   \label{Pennertt-fig}
\end{figure} 

Then $\tt$ lifts to a train track with oriented edges $\widetilde \tt$ on $\widetilde S$.
To understand the action of $\widetilde \phi$ it suffices to know where $\widetilde \phi$ sends each 
closed curve the lifts  $\widetilde a, \widetilde b$ and $\widetilde c$  of $a,b$ and $c$ on $\widetilde S$.
For example, $\delta_{\widetilde c}$ acts by the identity on $\widetilde a$ and 
$\widetilde c$ and $\delta_{\widetilde c} (\widetilde b)$ is shown in Figure~\ref{Pennerttcovering-fig}. 
(see Figure~\ref{Pennerttcovering-fig}).
\begin{figure}[htbp] 
   \centering
   \includegraphics[width=3in]{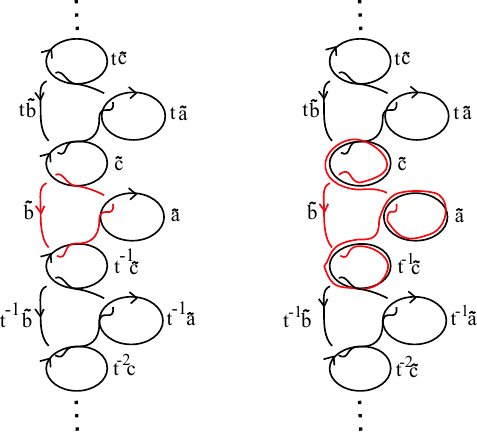} 
   \caption{Lift of $\tt$ to $\widetilde S$ and action of $\delta_{\widetilde c}$ on $\widetilde b$.}
   \label{Pennerttcovering-fig}
\end{figure}

When considering the algebraic action of $\widetilde \phi_0$, we count $\widetilde c^-$ as $-\widetilde c$.  Thus the
algebraic action of $\widetilde \phi$ on $(\Z[t,t^{-1}])^{\widetilde a,\widetilde b,\widetilde c}$ as the matrix
$$
\mathcal A = \left [
\begin{array}{ccc}
1 &  1 & 0\\
1 & 2 & 1-t \\
 1 - t^{-1} &2(1 - t^{-1}) & 1 + (1-t)(1-t^{-1})
 \end{array}
 \right ].
 $$
 The Alexander polynomial $\Delta$ is the characteristic polynomial of the action of
 the $\widetilde\phi$ on  $H_1(\widetilde S;\Z)$ as a $\Z[t,t^{-1}]$-module.   More precisely, there is a free presentation of
 $H_1(\widetilde M;\Z)$ as a $\Z[t,t^{-1}]$-module given by
 $$
\xymatrix{
(\widetilde \Z[t,t^{-1}])^{\widetilde  a,\widetilde b,\widetilde c}\ \ar[r]^{\mathcal A - u I} &\ (\widetilde \Z[t,t^{-1}])^{\widetilde a,\widetilde b,\widetilde c} \ar[r]
&H_1(\widetilde M;\Z),\\
}
$$
and the determinant of $\mathcal A - u I$ is the generator of the first fitting ideal for this presentation, that is, the {\it Alexander polynomial}
of $M$.  In our example, this is given by.
\begin{eqnarray*}\label{Alex-eqn}
 \Delta(u,t) = \Theta(u,-t) = u^2 - u(5 - t -  t^{-1}) + 1.
 \end{eqnarray*}
 
Considering $\widetilde a,\widetilde b$ and $\widetilde c$ as simple closed curves, we forget the sign of $\widetilde c^-$
and the action of $\widetilde \phi$ on $(\Z[t,t^{-1}])^{\widetilde a,\widetilde b,\widetilde c}$ becomes
 $$
\left [
\begin{array}{ccc}
1 &  1 & 0\\
1 & 2 & 1+t \\
 1 + t^{-1} &2(1 + t^{-1}) & 1 + (1+t)(1+t^{-1}).
 \end{array}
 \right ].
 $$
Thus the Teichm\"uller polynomial  is the characteristic polynomial of this matrix:
\begin{eqnarray*}
 \Theta(u,t) = u^2 - u(5 + t+  t^{-1} ) + 1.
\end{eqnarray*}

 \begin{remark} {\em In general, our group $H$ is a quotient group of $G$ isomorphic to $\Z \times \Z$, and $\Delta$ is the image of $\Delta_M$ in $\Z[H]$.
 Since in this example, $b_1(M) =2$,  $\Delta(u,t)$ is the Alexander polynomial $\Delta_M$ of $M$ and not a specialization.
  }
 \end{remark}
 
 \subsection{Fibered face}
The fibered face of a 3-manifold $M$ associated to a flow equivalence class can be found explicitly from 
the Tecihm\"uller polynomial of the face, and the Alexander polynomial of $M$ by a result of McMullen 
\cite{McMullen:Alex}, which we recall here.

Let $H$ be a finitely generated free abelian group.  Write $f \in  \Z H$ as
$$
f= \sum_{h \in H_0} a_h h
$$
where $H_0 \subset H$ is a finite subset, and $a_h \neq 0$ for all $h \in H_0$.   This representation for $f$
is unique, and we call $H_0$ the {\it support} of $f$.
If $H_0$ is in general position in $H \otimes \R$ considered as a Euclidean space,
then there is a corresponding norm on $\Hom(H;\R)$ given by
$$
||\alpha ||_f = \max \{|\alpha(h_1) - \alpha(h_2)| \ : \ h_1,h_2 \in H_0\},
$$
and the norm ball for $|| \ ||_f$ is convex polyhedral.

McMullen showed in \cite{McMullen:Alex} that if $F$ is a fibered face of a hyperbolic 3-manifold,  $\Delta$ and $\Theta_F$ 
are the Alexander and Teichm\"uller polynomials, and $b_1(M) \ge 2$, then the Thurston norm $|| \ ||$ restricted to the cone $V = F \cdot \R^+$
has the property that
$$
|| \alpha || = || \alpha ||_\Delta \leq || \alpha ||_{\Theta_F}
$$
for all $\alpha \in V_F$.  

\begin{lemma}\label{PennerCone-lem} The fibered cone $C$ in $H^1(M;\R)$ associated to Penner wheels
is given by elements $(a,b) \in H^1(M;\R)$, satisfying
 $$
a > | b |,
 $$
 and the Thurston norm is given by
 $$
 ||(a,b)||_T = \max \{2|a|, 2|b|\}.
 $$
 \end{lemma}

\subsection{Dilatations  and normalized dilatations}

The dilatation $\lambda(\phi_\alpha)$ corresponding to primitive
integral points $\alpha = (a,b)$ in $V$
is the largest solution of the polynomial equation
$$
\Theta(x^{a},x^{b}) = 0.
$$ 
In particular, 
Penner's examples $(R_g,\psi_g)$ correspond to the points $(g,1) \in V$, and 
we have the following.

\begin{proposition} The 
dilatation of $\psi_g$ is given by the largest root of the polynomial
$$
\Theta(x^g,x) = x^{2g} - x^{g+1} - 5x^{g} - x^{g-1} + 1.
$$
\end{proposition}

The limiting mapping class $(S_0,\phi_0)$ corresponds to the specialization 
$$
\theta(x) = \Theta(x,1) = x^2 - 7x +1,
$$
so $\lambda(\phi_0) = \frac{1}{2} (7 + 3\sqrt{5})\approx 6.8541$.

The symmetry of $\Theta$ with respect to $x \mapsto -x$  and convexity of $L$ on fibered
faces implies
 the minimum normalized dilatation realized on the fibered face
must occur at $(a,b) = (1,0)$, i.e., at th emonodromy $(S_0,\phi_0)$.  Thus, we have the following.

\begin{proposition} For all monodromies $(S',\phi')$ of primitive integral elements in the cone $V$,
$$
L(S',\phi') \ge L(S_0,\phi_0)  = \lambda(\phi_0)^2 \approx 46.9787.
$$
\end{proposition}

\subsection{Orientability}
A pseudo-Anosov mapping class is {\it orientable} if it has orientable invariant foliations, 
 or equivalently the geometric and homological dilatations are the same, and
 the spectral radius of the homological action is realized by a real (possibly negative)
 eigenvalue (see, for example, \cite{LT09} p. 5).   
 Given a polynomial $f$, the largest complex 
norm amongst its roots is called the {\it house of $f$}, denoted $|f|$.   
Thus, $\psi_g$ is orientable if and only if
\begin{eqnarray}
|\Delta(x^g,x))|= |\Theta(x^g,x)|.
\end{eqnarray}

\begin{proposition} The mapping classes $(R_g,\psi_g)$ are orientable if and only
if $g$ is even.
\end{proposition}

\begin{proof}   The homological dilatation of $\psi_g$
is the largest complex norm amongst roots of 
$$
\Delta(x^g,x) = x^{2g} + x^{g+1} - 5 x^g + x^{g-1}  + 1. 
$$
Let $\lambda$ be the  real root of $\Delta(x^g,x)$ with largest absolute value.
Plugging $\lambda$ into $\Theta(x^g,x)$ gives 
$$
\Theta(\lambda^g,\lambda) = -2\lambda^{g+1} - 2 \lambda^{g-1} \neq 0.
$$
while for $-\lambda$ we have
$$
\Theta(-\lambda^g,-\lambda) = (-\lambda)^{g+1} -( \lambda)^{g+1} + (-\lambda)^{g-1} - (\lambda^{g-1})
$$
which equals 0 if and only if $g$ is even.
\end{proof}

\subsection{Boundary behavior}

By Lemma~\ref{PennerCone-lem} we can extend the parameterization 
$$
\mathfrak f: I_2 = (0,\frac{1}{2}) \rightarrow F
$$
  to
\begin{eqnarray*}
\mathfrak f:  (-1,1) &\rightarrow& F\\
\r & \mapsto & \frac{1}{|\chi(\r)|}(1,\r),
\end{eqnarray*}

\begin{lemma}\label{revPenner-lem}
The sequence of mapping classes associated to  $\mathfrak f  (\frac{n-1}{n})$  is conjugate to 
$$
(\widetilde S/\widetilde \zeta^n, \zeta_n \delta_{\zeta_n^{-1}(c)} \delta_{b}^{-1} \delta_a).
$$
\end{lemma}

\begin{proof}
Let $R : \widetilde S \rightarrow \widetilde S$ be a rotation around an axis that passes through $\widetilde a \cup \widetilde b$
in 3 points, preserves each of $\widetilde a$ and $\widetilde b$, and $R \widetilde c = \zeta^{-1}(\widetilde c)$.
Then we have 
\begin{eqnarray*}
R^{-1}  \widetilde \zeta  R  &=& \widetilde \zeta^{-1}\\
R^{-1}  \delta_{\widetilde a}  R &=& \delta_{\widetilde a}\\
R^{-1}  \delta_{\widetilde b}  R &=& \delta_{\widetilde b}\\
R^{-1}  \delta_{\widetilde c} R &=&  \delta_{\widetilde \zeta^{-1}(\widetilde c)}^{-1}.
\end{eqnarray*}

We have
$$
\widetilde \zeta^{-1} \widetilde \eta = \widetilde \zeta^{-1} \delta_{\widetilde c} \delta_{\widetilde b}^{-1} \delta_{\widetilde a} 
=
\widetilde \zeta^{-1} R\delta_{\widetilde \zeta^{-1}(c)}R  \delta_{\widetilde b}^{-1} \delta_{\widetilde a} 
$$
Conjugating by $R$ we have
\begin{eqnarray*}
R\widetilde \zeta^{-1} \widetilde \eta R &=& R\widetilde \zeta^{-1} R\delta_{\widetilde \zeta^{-1}(c)}R \delta_{\widetilde b}^{-1} \delta_{\widetilde a} R\\
&=&\widetilde \zeta \delta_{\widetilde \zeta^{-1}(c)} \delta_{\widetilde b}^{-1} \delta_{\widetilde a}
\end{eqnarray*}
\end{proof}

By Lemma~\ref{revPenner-lem}, the mapping classes $(S_{\frac{n-1}{n}},\phi_{\frac{n-1}{n}})$, also known as the {\it reverse Penner sequence},
are the same as the mapping classes for
the sequence $\mathfrak (\frac{1}{n})$ in the
 family $Q'= Q(\widetilde S, \widetilde \zeta, \delta_{\zeta^{-1}(\widetilde c)} \delta_{\widetilde b}^{-1} \delta_{\widetilde a})$.
 In fact, $Q$ and $Q'$ are equal, but parameterized so that $\iota'(c) = \iota(1-c)$.
By Lemma~\ref{PennerCone-lem}, $Q'$ is not stable, and it follows that $\lim_{n \rightarrow \infty} L(S_{\frac{n-1}n},\phi_{\frac{n-1}n}) = \infty$.

\section{How common are quotient families?}\label{open-sec}
A set of pseudo-Anosov mapping classes $\sF \subset \sP$ has  {\it small dilatation} if there is some $L$ so that for each $(S,\phi) \in \sF$,
$L(S,\phi) < L$, or in other words, 
$$
\log (\lambda(\phi)) < \frac{ \log (L)}{|\chi(S)}.
$$

Govem a pseudo-Anosov mapping class $(S,\phi)$, let $(S^0,\phi^0)$ be the mapping class obtained  by removing any singularities of the 
invariant stable foliation.   Then $(S^0,\phi^0)$ is called the {\it fully-punctured representative} of $(S,\phi)$. 
Given a collection $\sF$ of pseudo-Anosov mapping classes, let $\sF^0$ be the collection of fully-punctured representatives of elements of $\sF$.
 One may ask if the following analog of the
Farb-Leininger-Margalit University Finiteness Theorem is true.

\begin{question} [Penner-wheel question]  If $\sF$ is a small dilatation collection of  pseudo-Anosov mapping classes, is $\sF^0$ contained in a finite union of 
(fully-punctured) quotient families?  
\end{question}

One can also weaken the question as follows.

\begin{definition}{\em For $\kappa > 0$, a
pseudo-Anosov mapping class $(S,\phi)$ is
{\it $\kappa$-quasi-periodic} if there is a subsurface $Y \subset S$  and a mapping class $\zeta: Y \rightarrow Y$
such that
\begin{enumerate} 
\item for some $m$, $\zeta^m$ is a product of Dehn twists along boundary parallel curves on $Y$ (i.e. $\zeta$ is {\it periodic rel boundary on $Y$}); and
\item the support of $\zeta\phi$ has topological Euler characteristic $\chi$ bounded by
$$
- \kappa < \chi< 0.
$$
\end{enumerate}
In other words, $(x,\phi)$ is $\kappa$-quasi-periodic if $\phi = \zeta \eta$ for some $\eta$ with $\kappa$-small support.

A mapping class $(S,\phi)$ is {\it strongly $\kappa$-quasi-periodic} if $Y = S$.
A family of mapping classes $\sF \subset \sP$
is  {\it (strongly) quasi-periodic}  if for some $\kappa > 0$ all its members are {\it (strongly) $\kappa$-quasi-periodic. }}
\end{definition}

\begin{question} [Ferris-wheel question]  If $\sF$ is a small dilatation collection of  pseudo-Anosov mapping classes, is $\sF^0$ contained in a quasi-periodic family?  
\end{question}

Penner-type sequences and quotient families of mapping classes are strongly  quasi-periodic, where we can take 
$\kappa = m_1|\chi(\Sigma)|$.  Small dilatations families such as those found in \cite{HK:braidbounds} and \cite{Hironaka:LT}
are quasi-periodic in the weaker sense, and have not been shown to be strongly quasi-periodic.

\bibliographystyle{../../math}
\bibliography{../../math}

\end{document}